\documentclass[10pt]{article}
\usepackage[english]{babel}
\usepackage{sansmathfonts}
\usepackage[T1]{fontenc}
\usepackage{amsmath,amsfonts,amssymb,amstext,amsmath,amsthm,amscd}
\usepackage{mathtools}
\usepackage[a4paper,top=2cm,bottom=2.5cm,left=2.5cm,right=2.5cm]{geometry}
\usepackage{breakcites}
\usepackage{mathtools}

\usepackage{enumitem}
\usepackage{booktabs}
\usepackage{mathtools}
\usepackage{mathrsfs}
\usepackage{dsfont}
\usepackage{enumitem}
\usepackage{graphicx}
\usepackage{subfigure}
\usepackage{wrapfig}
\usepackage{indentfirst}
\usepackage{latexsym}
\usepackage{sidecap}
\usepackage{booktabs}
\usepackage{hyperref}
\newcommand\setItemnumber[1]{\setcounter{enumi}{\numexpr#1-1\relax}}
  
  \usepackage{lipsum}
  \usepackage{cleveref}
  \usepackage[dvipsnames]{xcolor}
  \usepackage[nodisplayskipstretch]{setspace} 
  \newcommand\xqed[1]{%
    \leavevmode\unskip\penalty9999 \hbox{}\nobreak\hfill
    \quad\hbox{#1}}
      \newcommand\demo{\xqed{$\triangle$}}
      
      \newcommand\myshade{85}
      \colorlet{mylinkcolor}{violet}
      \colorlet{mycitecolor}{YellowOrange}
      \colorlet{myurlcolor}{RoyalBlue}
      
      \hypersetup{
        colorlinks=true,
        linkcolor=myurlcolor,
        citecolor=mycitecolor!\myshade!black,
        filecolor=magenta,      
        urlcolor=magenta,
      }
      \usepackage{setspace}
      \usepackage{breakcites}
      
      \usepackage{mathrsfs}
      \usepackage{textcomp}
      \usepackage{braket}
      \usepackage{colortbl}
      \usepackage{bbm}
      \usepackage{comment}
      \usepackage[colorinlistoftodos]{todonotes}\usepackage{mathrsfs}
      \usepackage{textcomp}
      \usepackage{braket}
      \usepackage{colortbl}
      \usepackage{bbm}
      \usepackage{comment}
      \usepackage[colorinlistoftodos]{todonotes}
      
      \theoremstyle{plain}
      \newtheorem{Theorem}{Theorem}[section]
      \theoremstyle{definition}
      \newtheorem{Definition}{Definition}[Theorem]
      \theoremstyle{remark}
      \newtheorem{Remark}[Theorem]{Remark}
      \theoremstyle{remark}
      
      \theoremstyle{plain}
      \newtheorem{Lemma}[Theorem]{Lemma}
      \theoremstyle{plain}
      
      \theoremstyle{plain}
      \newtheorem{Proposition}[Theorem]{Proposition}
      \theoremstyle{remark}
      
      \theoremstyle{remark}
      
      \theoremstyle{remark}
      \newtheorem{Hypothesis}[Theorem]{Hypothesis}
      \theoremstyle{remark}

      \newcommand\RR{\mathbb{R}}
      \newcommand\NN{\mathbb{N}}

      \newcommand\KK{\mathrm{K}}

\newcommand{\ii}{\mathbf{i}}

\DeclareMathAlphabet{\pazocal}{OMS}{zplm}{m}{n}

\newcommand{\CC}{\mathcal{C}}
\newcommand{\DD}{\mathcal{D}}
\newcommand{\XX}{\mathrm{X}}
\newcommand{\YY}{\mathrm{Y}}
\newcommand{\QQ}{\mathrm{Q}}
\newcommand{\Sp}{\mathrm{S}}
\newcommand{\PP}{\mathrm{P}}
\newcommand{\EE}{\mathbb{E}}
\newcommand{\ZZ}{\mathrm{Z}}
\newcommand{\LL}{\mathcal{L}}
\newcommand{\norm}[1]{ \left\langle #1 \right\rangle}
\newcommand{\unit}{\mathbf{1}}
\newcommand{\HH}{\mathrm{H}}
\newcommand{\Mar}{\mathrm{M}}
\newcommand{\GSM}{GSM$^2$}

\usepackage{tikz}
\usetikzlibrary{arrows}

\newcommand{\Y}{\mathbf{Y}}

\newcommand{\Meas}{\mathcal{M}_F}

\newcommand{\Ind}[1]{\mathbbm{1}_{\left \{#1\right \}}}

\newcommand{\MM}{\mathcal{M}}
\newcommand{\MMPP}{\mathcal{M}_{\mathrm{F}}}
\newcommand{\RG}{\mathrm{r}}
\newcommand{\AG}{\mathrm{a}}
\newcommand{\BG}{\mathrm{b}}
\newcommand{\PN}{\mathrm{N}}
\newcommand{\PNC}{\tilde{\mathrm{N}}}
\newcommand{\DG}{\dot{\mathrm{d}}}
\newcommand{\PG}{\mathrm{p}}
\newcommand{\NXK}{\nu^{\mathrm{X};\mathrm{K}}}
\newcommand{\NYK}{\nu^{\mathrm{Y};\mathrm{K}}}

\usepackage{tikz}
\usetikzlibrary{arrows}

\title{A spatial measure-valued model for radiation-induced DNA damage kinetics and repair under protracted irradiation condition}
\author{Francesco G. Cordoni}
\date{\today}

\begin{document}

\maketitle   

\begin{abstract}
In the present work, we develop a general spatial stochastic model to describe the formation and repair of radiation-induced DNA damage. The model is described mathematically as a measure-valued particle-based stochastic system and extends in several directions the model developed in \cite{cordoni2021generalized,cordoni2022cell,cordoni2022multiple}. In this new spatial formulation, radiation-induced DNA damage in the cell nucleus can undergo different pathways to either repair or lead to cell inactivation. The main novelty of the work is to rigorously define a spatial model that considers the pairwise interaction of lesions and continuous protracted irradiation. The former is relevant from a biological point of view as clustered lesions are less likely to be repaired, leading thus to cell inactivation. The latter instead describes the effects of a continuous radiation field on biological tissue. We prove the existence and uniqueness of a solution to the above stochastic systems, characterizing its probabilistic properties. We further couple the model describing the biological system to a set of reaction-diffusion equations with random discontinuity that model the chemical environment. At last, we study the large system limit of the process. The developed model can be applied to different contexts, with radiotherapy and space radioprotection being the most relevant. Further, the biochemical system derived can play a crucial role in understanding an extremely promising novel radiotherapy treatment modality, named in the community \textit{FLASH radiotherapy}, whose mechanism is today largely unknown.
\end{abstract}
\tableofcontents

\section{Introduction}

Radiotherapy is, today, a widely used treatment against cancer \cite{thariat2013past}. Conventional radiotherapy is based on X-rays, i.e. photons, but in the last decades constantly increasing attention has been devoted to advanced radiotherapy treatment with ions \cite{durante2016nuclear}. Ion beams have many essential features making them preferable compared to photons, related mostly to the extremely localized energy released in tissues which can lead to a superior biological effect than X-rays. The effect of radiation on biological tissue has been studied by the community over the last decades, and DNA is believed to be the most sensitive target to radiation so DNA damage is the most relevant biological vehicle that leads to cell killing induced by radiation, \cite{durante2010charged}. Despite the potential superiority of hadrons in theory, additional research is crucial to incorporate this treatment modality into clinical practice fully. One of the primary obstacles to the widespread use of hadrons is in fact accurately estimating the biological effect caused by radiation, a crucial aspect to account for in order to prescribe the best possible treatment. Mathematical models have thus been developed over the years to understand and accurately predict the biological effect of ions on biological tissue, \cite{bellinzona2021linking,hawkins1994statistical,hawkins2013microdosimetric,kellerer1974theory,herr2015comparison,pfuhl2020prediction,cordoni2021generalized}, focusing on the DNA damage Double Strand Breaks (DSB). Such mathematical approaches focus on developing models that describe the formation, evolution, and interaction of DSB, with the final goal of predicting the probability that a certain cell survives radiation.

To date, very few models in the context of radiotherapy have a robust mathematical and probabilistic background even if the community widely acknowledges stochastic effects play a major role in the biological effect of radiation. In fact, despite the early development of stochastic models for the description of the kinetic repair of radiation-induced DNA damages, \cite{sachs1990incorporating,albright1989markov}, the radiobiological community soon drifted to developing deterministic models of damage repair assuming Poisson fluctuations of the number of damages around the average values, \cite{bellinzona2021linking}. This type of modelization is strictly linked to a linear-quadratic description of the relation between the logarithm of the cell-survival probability and the absorbed dose, a physical quantity that describes the energy deposited by the particles over the mass of the biological tissue traversed by the particles. Although such models provide a fast way to assess the cell survival fraction, which is a key aspect for a use of such models in clinical applications in which the run time of a model is extremely relevant, in recent years, the need began to be felt for more robust modeling from a purely probabilistic point of view. From a mathematical point of view, the Generalized Stochastic Microdosimetric Model (GSM$^2$) recently introduced in \cite{cordoni2021generalized,cordoni2022cell,cordoni2022multiple}, appears to be a general mathematical model, that includes several relevant stochastic effects emerging in the creation, repair, and kinetics of radiation-induced DNA-damages, \cite{cordoni2022multiple}. GSM$^2$ considers two types of DNA lesions $\XX$ and $\YY$, representing respectively lesions that can be repaired and lesions that lead to cell inactivation. In the current context, the specific exact meaning of sub-lessons is left unspecified. This is because there are mainly two different ways that cells can be affected by radiation. One is the creation of DNA \textit{Double-Strand Breaks} (DSB) from two \textit{Single-Strand Breaks} (SSB), and the other is the formation of chromosome abnormalities from pairs of chromosome breaks, \cite{kellerer1974theory}. Both of these mechanisms are important in understanding how cells respond to radiation and can be described by the model developed in this work.

The present paper aims at extending GSM$^2$ to include a spatial description allowing for reaction rates that depend on the spatial position, lesion distance, and density. In fact, a true spatial distribution of DNA damage inside the cell nucleus is today almost completely missing in existing models. At the same, it is widely known in the community that the spatial distribution of DNA damages strongly affects the probability that a cell repairs the induced damages, so spatial stochastic plays a major role in the modelization of the repair of radiation-induced DNA damages. We will thus model the spatial distribution of DNA damages as a general measure-valued stochastic particle-based system, characterizing existence and uniqueness as well as some relevant martingale properties that, as standard, will play a crucial role in the derivation of the large system limit. Stochastic particle-based systems have been long studied in the mathematical community, \cite{bansaye2015stochastic,popovic2011stochastic,pfaffelhuber2015scaling}. Recently, a lot of attention has been devoted to studying the spatial non-local stochastic particle-based system, \cite{bansaye2010branching,fontbona2015non,champagnat2007invasion,ayala2022measure}, where a measure-valued stochastic process describes the population. The population can interact according to a specific rate leading to either the creation or removal of individuals. Mathematically, these systems are described by Stochastic Differential Equations (SDE) driven by Lévy-type noises that besides a diffusive component include jump operators in the form of Poisson random measure, that account for the creation and removal of individuals from the population, \cite{bansaye2015stochastic}. Most results focus on birth-and-death spatial processes, meaning that at each time, at most, a single individual can be born or die. In this setting, pairwise interactions, involving either the creation or removal of more than one individual, are not allowed. Such interactions are relevant in many biological and chemical applications so a general mathematical theory that extends and generalizes the birth and death process is greatly desirable. Recently, few papers appeared that include pairwise reactions, \cite{isaacson2022mean,lim2020quantitative,popovic2020spatial}, but none of these deal with the existence and uniqueness or regularity of results for such systems, where instead the focus is mostly on the large-population limit.

The developed model includes some key features that make the mathematical treatment of the spatial model non-trivial. First, given the application considered, where clusters of DNA lesions are more difficult to repair by the cell and has been recognized as one of the main factors that lead to cell inactivation in radiobiology, \cite{kellerer1974theory}, we will include pairwise interaction and second-order rates, meaning that a couple of lesions can interact to create an unrepairable lesion that inactivates the cell. It is worth stressing that, already, many existing radiobiological models include parameters to account for the interaction of damages, \cite{hawkins1994statistical,sato2012cell,hawkins2013microdosimetric,bellinzona2021linking,cordoni2021generalized,cordoni2022multiple}. Still, none have a true mathematical spatial formulation and often rely on fixed domains to limit pairwise interaction within a certain distance neglecting nonetheless any true spatiality inside a fixed domain. The latter approach can be restrictive and may lead to overfitting with the inclusion of unnecessary parameters. One of the proposed model's main strengths is that it considers a true spatial distribution of lesions, allowing for true pairwise interaction that can depend on the distance between lesions, which is a novel and important aspect of the model. As mentioned, the existence and uniqueness results for second-order systems are rare and yet a general theory is missing, so the derived results represent a novelty both from a radiobiological as well as a mathematical perspective.

Another key aspect of the studied model is that we explicitly consider the case of protracted irradiation, that is we consider the situation in which a continuous radiation field induces a random number of lesions in the cell. Such a situation is non-trivial from a purely mathematical perspective as the generation of a random number of damages must be considered. Nonetheless, it is extremely relevant to include protracted irradiation in a biological model since it allows us to better estimate the kinetics repair of radiation-induced damage with benefits both in radioprotection and clinical application. Existing radiobiological models account to a certain extent for the protracted irradiation case, \cite{inaniwa2013effects,manganaro2017monte}, and a similar rate has already been considered in \cite{cordoni2021generalized} in a non-spatial setting. In generalizing the setting to include a spatial description, we need to describe a spatial energy deposition pattern within the cell nucleus. We will build such a theory properly generalizing some existing approaches. Nonetheless, a robust theory to account for the spatial formation of radiation-induced DNA lesions is missing and future efforts will be made to derive such a theory. In fact, another relevant aspect of the studied model is the computation of the spatial distribution of radiation-induced DNA damages. Such a description is relevant both in the protracted case and in the instantaneous irradiation case as it describes the initial damage distribution. In particular, the initial damage distributions, $\nu^\XX_0$, and $\nu^\YY_0$ can be computed using different methods. One possible approach is the use of Monte Carlo (MC) track structure codes, \cite{nikjoo2006track}, to simulate the passage of charged particles in biological tissue and their energy release and to thus estimate the DNA damage distribution caused by radiation. MC track structure codes have been shown to be effective in accurately characterizing DNA damage formation, \cite{goodhead1994initial,ottolenghi1995quality,a2000kinetics,chatzipapas2022simulation,kyriakou2022review,zhu2020cellular,thibaut2023minas}, however, once the initial damage distribution is computed, in order to assess the cell survival probability, these models typically neglect the spatial distribution of damages and focus on average values described by \textit{Ordinary Differential Equations} (ODE). The model developed in this research is unique in that it is able to fully exploit the accuracy of the spatial distribution of damages as predicted by MC track-structure codes. Further, since MC track structure codes simulate all the energy released by a particle along its path, which is referred to in the community as track, as well as all secondary energy releases associated with the original particle, the computational time is extremely demanding. To shorten the computational time, a threshold on energy release can be applied. so that all events that release lower than a certain energy are neglected and incorporated into the deposition that has originated it. Such an approach is called \textit{condensed history} MC, \cite{agostinelli2003geant4}, and it provides accurate results of energy deposition at a lower computational time compared to MC track structure codes.

An alternative approach to MC track-structure codes would be to develop an analytical model for DNA damage formation and distribution. Such a model would be less accurate, but less computationally expensive. Currently, the \textit{Local Effect Model} (LEM), \cite{friedrich2012calculation}, and the \textit{Microdosimetric Kinetic Model} (MKM), \cite{hawkins1994statistical,kase2007biophysical}, which are the only models used in \textit{Treatment Planning Systems} (TPS), take into account the spatial distribution of the absorbed dose without MC codes. These models are based on the \textit{Amorphous Track} (AT) model, \cite{kase2007biophysical}, which parametrizes the dose distribution around a track of a particle. However, to eventually assess the cell-survival probability, both models make extensive use of fixed domain so that a true spatial distribution of damages is again neglected. It is worth further stressing that, although the AT model can be used to compute the imparted dose in a fast fashion, it is based on several assumptions, such as the so-called track-segment condition (assumes tracks do not lose energy when traversing the cell nucleus) and uniform radiation fields (cylindrical geometry is often assumed for the cell nucleus and tracks are perpendicular to the cell nucleus). Being the former approach based on track-structure codes beyond the scope of the present work, we will focus on the latter one. Nonetheless, future research will focus on developing a comprehensive analytical model for DNA damage formation that accurately describes energy and spatial stochasticity.

The developed spatial DNA-damage model is expected to play a relevant role in the modelization of a novel radiotherapeutic technique, named in the community FLASH radiotherapy \cite{favaudon2014ultrahigh}. Up to 2014, most of the mechanisms happening in the interaction of radiation with biological tissue were believed to be known, and therefore in the last two decades models focused on specific applied aspects and general and robust mathematical theories were strongly believed to be unnecessary given the overall understanding of the problem at hand. Starting in 2014 a series of ground-breaking papers \cite{favaudon2014ultrahigh,montay2017irradiation,vozenin2019advantage} finally showed that an increase in the rate of delivery of ions at Ultra-High Dose Rates (UHDR), namely a high amount of energy released in a small fraction of time, spares healthy tissue and yet maintains the same effect on the tumor. This effect, which was completely unexpected, represents the final goal of any radiotherapeutic treatment \cite{esplen2020physics,griffin2020understanding}. All available models brutally failed to predict such peculiar effects and hundreds of publications have appeared recently trying to understand the mechanism at the very core of the FLASH effect \cite{labarbe2020physicochemical,abolfath2020oxygen,abolfath2020dna,liew2021deciphering,petersson2020quantitative}; many physical explanations and mathematical models have been proposed in the last five years,  but up to date, no model is believed to be capable neither of predicting nor understanding the origin of the FLASH effect \cite{weber2022flash}. Two facts are today believed to be at the very core of the FLASH effect: (i) this effect has its origin in a spatial interaction of ions, that involves besides the physics of ions and biology, also chemistry, and (ii) nonlocal effects of ions while traversing different cells and how their spatial interaction affects the overall chemical and biological environment. The model developed in the present research could, when coupled with an adequate description of the chemical environment affected by radiation, help unravel the mechanism behind the FLASH effect. In order to do that, we couple the spatial model with a reaction-diffusion equation that describes the evolution of the chemical system. In the following treatment, we will not specify a particular chemical description. This is because there are several possible choices, and the choice depends on the specific application. The chemical stage can be broadly divided into two stages: (i) the heterogenous chemical stage and (ii) the homogeneous chemical stage. The former is characterized by a heterogeneous spatial distribution and occurs immediately after particles hit the cell nucleus, i.e. between $10^{-12}$ seconds and $10^{-6}$ seconds. The latter is characterized by a homogeneous distribution and occurs after the heterogeneous stage, i.e. between $10^{-6}$ seconds to $10^0$ seconds. Mathematically, this means that while the homogeneous stage can be described by ODEs that characterize the time evolution of the concentration of chemicals within the domain, \cite{labarbe2020physicochemical,abolfath2020oxygen}, the heterogeneous stage requires advanced mathematical tools since the system is highly nonlinear and the reactions occur locally. Therefore, most of the literature is devoted to the development of simulation codes,  \cite{clifford1986stochastic,pimblott1991stochastic,boscolo2020impact,ramos2020let}. To date, there is no general mathematical formulation via local reaction-diffusion PDEs that exists in literature, even though it could provide an accurate representation of the system and fast computational time. Also, general results of well-posedness that cover relevant non-local chemical systems are not available in the literature due to the highly complex mathematical formulation needed. For these reasons, a deep mathematical study of such a system is left for future research. Relevant systems to study could include a spatial non-local version of the homogeneous chemical systems presented in \cite{labarbe2020physicochemical,abolfath2020oxygen} or an analytical formulation in terms of highly dimensional non-local reaction-diffusion PDEs of \cite{boscolo2020impact,ramos2020let}. In this paper, we will instead limit ourselves to considering a general reaction-diffusion PDE with coefficients satisfying rather general assumptions that could in principle include several relevant examples. In particular, almost any chemical description includes bimolecular reactions, meaning that the resulting PDE has quadratic terms. A significant effort has been made in the literature to study reaction-diffusion PDEs under the most general assumptions on the coefficients in order to include as many examples as possible, \cite{pierre2010global}. In this direction, a mass control assumption has been typically seen as a general condition that allows obtaining the existence and uniqueness of the equations in many cases dropping the standard global Lipschitz condition. For the sake of simplicity, we will consider the system introduced in \cite{fellner2020global}, which allows for quadratic growth of the coefficients, adding random discontinuity due to the effect of radiation. Therefore, we prove the existence and uniqueness of a stochastic particle-based system coupled with a reaction-diffusion system with random jumps. We will show later that this system allows generalizing some relevant systems. A future effort will be devoted to the study of general local reaction-diffusion systems similar to the ones studied in \cite{isaacson2022mean}.

At last, we will also characterize the large-system behavior. Such a limiting system can be obtained with standard arguments proving the tightness of the measure and identifying the listing process, which can be proved to admit a unique solution. Although the techniques are standard, the result is new in literature since no stochastic system allowing for pairwise interaction and creation of random numbers of particles has never been studied before. In fact, The resulting governing equation can be useful to study the behavior of the system at high doses, where the number of damages within the cell increases arbitrarily. It is worth stressing that, the high-dose case is recognized to be non-trivial and most of the existing models fail to predict the behavior of the system at high doses. For this reason, often, correction terms are included in the model to better math experimental data, \cite{bellinzona2021linking}.

The main contributions of the present paper are:
\begin{description}
\item[(i)] to provide a general mathematical description of a spatial model governing the formation and kinetics of radiation induced-damages;\\
\item[(ii)] to study the well-posedness of a measure-valued stochastic particle system with pairwise interaction and random creation of damages;\\
\item[(iii)] to propose a multi-scale model to couple biology and chemistry that could possibly describe the FLASH effect;\\
\item[(iv)] to study the large-system limit of the system with pairwise interaction and protracted irradiation.
\end{description}

\section{The microdosimetric master equation}\label{SEC:GSM2}

The main goal of the present section is thus to introduce the classic setting for the GSM$^2$ \cite{cordoni2021generalized,cordoni2022cell}. GSM$^2$ models the time-evolution of the probability distribution of the number of lethal and sub-lethal lesions denoted by $\left (\XX(t),\YY(t)\right )$, where $\XX$ and $\YY$ are two $\mathbb{N}-$valued random variables counting the number of the lethal and sub-lethal lesion, respectively. In the following, we will consider a standard complete filtered probability space $\left (\Omega,\mathcal{F},\left (\mathcal{F}_t\right )_{t \geq 0},\mathbb{P}\right )$ satisfying usual assumptions, namely right--continuity and saturation by $\mathbb{P}$--null sets.

We thus assume that a sub-lethal lesion $\XX$ can undergo three different pathways: (i) at rate $\RG$ a sub-lethal lesion is repaired, (ii) at rate $\AG$ a sub-lethal lesion is left unrepaired by the cell and thus it becomes a lethal lesion and (iii) at rate $\BG$ two sub-lethal lesion form a cluster that cannot be repaired by the cell and thus become a lethal lesion. Any lethal lesion leads to cell inactivation. These three pathways can be summarized as follows
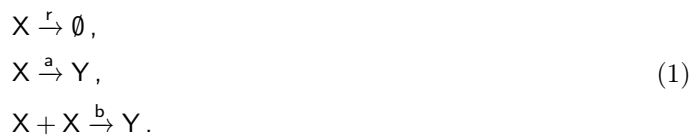
\begin{equation}\label{EQN:Reac}
\begin{split}
& \XX \xrightarrow{\RG} \emptyset\,,\\
& \XX \xrightarrow{\AG} \YY\,,\\
& \XX + \XX \xrightarrow{\BG} \YY\,.\\
\end{split}
\end{equation}

Denoting by
\[
p(t,y,x) = \mathbb{P}\left (\left (\YY(t),\XX(t)\right ) = \left (y,x\right )\right )\,,
\] 
the probability to have at time $t$ exactly $x$ sub-lethal lesion and $y$ lethal lesions, following \cite{cordoni2021generalized}, we can obtain the \textit{microdosimetric master equation} (MME)
\begin{equation}\label{EQN:Master}
\begin{cases}
\frac{\partial}{\partial t} p(t,y,x) &= \left (E^{0,1} -1\right )\left [x r p(t,y,x)\right ] + \left (E^{-1,1} -1\right )\left [x \AG p(t,y,x)\right ] + \left (E^{-1,2} -1\right )\left [x(x-1) \BG p(t,y,x)\right ] \,,\\
p(0,y,x) &= p_0(y,x)\,,
\end{cases}
\end{equation}
where we have denoted the creation operator as
\[
\left (E^{i,j}-1\right )\left  [f(y,x)\right  ] := f(y+i,x+j) - f(y,x)\,.
\]

In \cite{cordoni2022cell} it is derived a closed-form solution for the survival probability as predicted by the MME \eqref{EQN:Master}, defined as the probability of having no lethal lesions $\YY$. Further, GSM$^2$ is closely connected with one of the most used radiobiological models to predict the survival probability of cell nuclei when exposed to ionizing radiation, that is the Microdosiemtric Kinetic Model (MKM) \cite{hawkins1994statistical}. The main equations of the MKM describe the time-evolution of the average value $\bar{y}$, resp. $\bar{x}$, of the number of lethal, resp. sub-lethal, lesions, and are given by
\begin{equation}\label{EQN:LQM}
\begin{cases}
\frac{d}{dt} \bar{y}(t) =  \AG \bar{x} + b \bar{x}^2\,,\\
\frac{d}{dt} \bar{x}(t) = - (\AG+\RG) \bar{x} - 2 \BG \bar{x}\,.\\
\end{cases}
\end{equation}
The model further assumes that $\bar{y}$ is the average of a Poisson random variable so that by describing the average values we have complete knowledge of all the moments.

To obtain a suitable analytical solution to the equations \eqref{EQN:LQM}, it is often assumed that $(\AG + \RG) \bar{x} >> 2\BG \bar{x}$, so that above equation is reduced to
\begin{equation}\label{EQN:LQM2}
\begin{cases}
\frac{d}{dt} \bar{y}(t) =  \AG \bar{x} + b \bar{x}^2\,,\\
\frac{d}{dt} \bar{x}(t) = - (\AG+\RG) \bar{x} \,.\\
\end{cases}
\end{equation}
This highlights why in the high dose case the MKM must be corrected including additional terms. In fact, even if it is typically true that $(\AG + \RG) >> 2\BG$, at sufficiently high doses, the number of lesions $\bar{x}$ increases so that $(\AG + \RG) \bar{x}$ does not dominate anymore $2\BG \bar{x}$ and therefore the omission of the term in equation \eqref{EQN:LQM2} becomes non-negligible.

Further, it has been shown in \cite{cordoni2021generalized}, that the average of the MME coincides with the MKM equations \eqref{EQN:LQM} under a suitable \textit{mean-field assumption}, that us
\[
\mathbb{E}\left[\XX(t)(\XX(t)-1)\right ] \approx \mathbb{E}\left[\XX(t)\right ]^2\,,
\]
which in turn coincides exactly with the requirement that $\XX$ follows a Poisson distribution. It has thus been shown in \cite{cordoni2022multiple} that the GMS$^2$ is able to give a more general description of many stochastic effects relevant in the formation and repair of radiation-induced DNA lesions that play a crucial role in estimating the surviving probability of a cell nucleus.

It can be further shown, \cite{weinan2021applied,bansaye2015stochastic}, that equation \eqref{EQN:Master} describes the time evolution for the probability density function associated with the following \textit{stochastic differential equation} (SDE)
\begin{equation}\label{EQN:SDE}
\begin{cases}
\YY(t) &= \YY_0 + \int_0^t \int_{\RR_+} f^{\YY}(\XX(s^-),z) N^{\YY}(ds,dz)\,,\\
\XX(t) &= \XX_0 - \int_0^t \int_{\RR_+} f^{\XX}(\XX(s^-),z) N^{\XX}(ds,dz)\,,\\
\end{cases}
\end{equation}
with
\begin{equation}\label{EQN:SDECoeff}
\begin{split}
f^{\YY}(\XX(s^-),z) &= \Ind{z \leq \AG \XX(s^-)} + \Ind{\AG \XX(s^-) \leq z \leq \AG \XX(s^-) + \BG \left (\XX\right )^2(s^-)}\,\\
f^{\XX}(\XX(s^-),z) &= \Ind{z \leq (\AG+\RG) \XX(s^-)} + 2 \Ind{(\AG + \RG) \XX(s^-) \leq z \leq (\AG + \RG) \XX(s^-) + \BG \left (\XX\right )^2(s^-)}\,.
\end{split}
\end{equation}

Above in equation \eqref{EQN:SDE}, $N^Y(ds,dz)$ and $N^X(ds,dz)$ are two independent Poisson point measure with intensity $ds\,dz$ on $\RR_+ \times \RR_+$, see, e.g. \cite{applebaum2009levy}. The main of the present work will be to provide a spatial description of the SDE \eqref{EQN:SDE} so that $\XX$ and $\YY$ are replaced by random measures.

\subsection{On the initial distribution}\label{SEC:Init}

In order to later generalize the initial damage distribution, we introduce in the current section the distribution introduced in \cite{cordoni2021generalized,cordoni2022cell,cordoni2022multiple}. For a detailed treatment, we refer the interested reader to the mentioned papers or to \cite{bellinzona2021linking}.

Among the most powerful approaches to describe the formation of DNA lesions is using microdosimetry \cite{zaider1996microdosimetry}. Microdosimetry is the branch of physics that investigates the energy deposition in domains comparable to cell nuclei, that is of the order of some microns. At that scale, energy deposition is purely stochastic, so the main object used in microdosimetry are random variable and their corresponding distributions. Over the years many models have been developed based on microdosimetric principles, \cite{kellerer1974theory,zaider1996microdosimetry}, and both the MKM and GSM$^2$ assess the formation of DNA lesions using microdosimetry, \cite{hawkins1994statistical,bellinzona2021linking,cordoni2021generalized,cordoni2022cell}.

The main microdosimetric quantity of interest from the point of view of radiobiological models is the \textit{specific energy} $z$, \cite{zaider1996microdosimetry}. The \textit{specific energy} $z$ is the ratio between energy imparted by a finite number of energy depositions $\varepsilon$ over the mass $m$ of the matter that has received the radiation, that is
\[
z = \frac{\varepsilon}{m}\,.
\]

The stochastic nature of $\varepsilon$ implies that also $z$ is inherently stochastic. The single--event distribution $f_{1}(z)$ of energy deposition on a domain, \cite{zaider1996microdosimetry}, is the probability density distribution describing the energy deposition due to a single event, typically a particle traversing the domain. Such distribution is associated with a random variable $Z$ that describes the specific energy imparted on a certain domain of mass $m$. The average values of the random variable $Z$, referred to in the literature as \textit{fluence-average specific energy}, that is the mean specific energy deposition, is typically denoted in literature as $z_F$. By additivity property, the specific energy distribution resulting from $\nu$ tracks can be computed convolving $\nu$ times the single event distribution, \cite{zaider1996microdosimetry}. Therefore, the distribution $f_{\nu}$ of the imparted energy $z$ is computed iteratively as
\[
\begin{split}
f_{2}(z) &:= \int_0^\infty f_{1}(\bar{z})f_{1}(z-\bar{z})d\bar{z}\,,\\
&\dots\,,\\
f_{\nu}(z) &:= \int_0^\infty f_{1}(\bar{z})f_{\nu-1}(z-\bar{z})d\bar{z}\,.\\
\end{split}
\] 

We denote by $p_e(\nu|D,z_F)$ a discrete probability density distribution denoting the probability of registering $\nu$ events. Typically such distribution is assumed to be dependent on the total dose absorbed by the mass and the fluence of the incident particles. The standard assumption is that, since events are in a microdosimetric framework assumed to be independent, the distribution $p_e$ is a Poisson distribution of average $\frac{D}{z_F}$ so that we have
\[
p_e(\nu|D,z_F) :=  \frac{e^{- \frac{D}{z_F}}}{\nu!}\left (\frac{D}{z_F}\right )^{\nu}\,.
\]

Therefore, microdosimetry postulates that the actual energy deposition on a certain domain can be obtained via the \textit{multi-event specific energy distribution}
\[
f(z|D) := \sum_{\nu = 0}^\infty \frac{e^{- \frac{D}{z_F}}}{\nu!}\left (\frac{D}{z_F}\right )^{\nu} f_{\nu;d}(z)\,.
\]

At last, given a certain specific energy deposition $z$ by $\nu$ events, the induced number of lethal and sub-lethal lesions is again a random variable, with a discrete probability density function denoted by $p$. In general the average number of lethal, resp. sub-lethal, lesions is assumed to be a function of $z$, namely $\kappa(z)$, resp. $\lambda(z)$. Again, by independence on the number of created lesions, such distribution is assumed to be a Poisson distribution. Overall, the probability of inducing $x$ sub-lethal and $y$ lethal lesions can be computed as, \cite{cordoni2021generalized},
\begin{equation}\label{EQN:InitDistGSM2}
p_0(x,y) = \sum_{\nu = 0}^\infty \int_0^\infty p(x,y|z) p_e(\nu|D,z_F) f_{\nu}(z) dz \,,
\end{equation}
or assuming Poissonian distributions
\begin{equation}\label{EQN:InitDistGSM2Pois}
p_0(x,y) = \sum_{\nu = 0}^\infty \int_0^\infty e^{-\bar{\kappa}(z)}\frac{\left(\bar{\kappa}(z)\right)^x}{x!}e^{-\bar{\lambda}(z)}\frac{\left(\bar{\lambda}(z)\right)^y}{y!}\frac{e^{- \frac{D}{z_F}}}{\nu!}\left (\frac{D}{z_F}\right )^{\nu} f_{\nu}(z) dz \,,
\end{equation}
for suitable functions $\kappa(z)$ and $\lambda(z)$. These quantities summarize the free-radical reactions that result in a lesion. It is a function of the type of ionizing particle, details of the track structure, radical diffusion, and reaction rates, the point in the cell cycle, and the chemical environment of the cell. In the following, we will explicitly model these functions so that they depend on chemical concentration.

The classical assumption, which has been also considered in \cite{cordoni2021generalized}, is to assume such functions to be linear in $z$. Notable enough, it has been shown in \cite{cordoni2022cell} that, also assuming a Poissonian distribution on both $p_e$ and $p$, the resulting discrete probability density function \eqref{EQN:InitDistGSM2Pois} is not a Poisson distribution; as a matter of a fact, it has been shown to be a microdosimetric extension of the so-called \textit{Neyman distribution}, \cite{neyman1939new}, which is a well-known distribution in radiobiological modeling to treat the number of radiation-induced DNA damages. To have a better grasp on the distribution \eqref{EQN:InitDistGSM2}, it can be described by a stochastic chain of interconnected events: (i) given a certain dose $D$ and fluence average specific energy $z_F$, a given random number of events $\nu$ is registered in a cell nucleus; then (ii) such $\nu$ events deposits a certain random specific energy $z=z_1+\dots+z_\nu$. At last, (iii) the specific energy deposited $z$ induces a random number of lethal and sub-lethal lesions $y$ and $x$.

\section{The spatial radiobiological model}\label{SEC:SpatGSM2}

The current section aims at generalizing the radiobiological model as introduced in Section \ref{SEC:GSM2} to consider a spatial measure-valued process. Consider a closed bounded regular enough domain $\QQ \subset \RR^d$, $d \geq 1$, which should represent a cell nucleus. We assume that $\QQ$ has a smooth boundary $\partial \QQ$, and denote by $n(q)$ the outward normal direction to the boundary $\partial \QQ$ at the point $q$.

We consider two possible types of DNA damage, $\Sp = \{\XX,\YY\}$, where $\XX$ denotes sub--lethal lesions and $\YY$ are lethal lesions. We assume sub-lethal and lethal lesions can undergo three different pathways, $a$, $b$, and $r$, as introduced in Section \ref{SEC:GSM2}.

We consider thus a process that lives in the state space 
\[
\PP := \QQ \times \Sp \ni P_i = \left (q_i,s_i\right)\,,
\]
encoding the i-th lesion position $q_i$ and type $s_i$. For a metric space $E$, we define by $\Meas(E)$ the space of finite measure over $E$, endowed with the weak topology; given a regular enough function $f : E \to \RR$ and a measure $\nu \in \Meas(E)$, $\Meas(E)$ is equipped with 
\[
\norm{f,\nu}_E:=\int_E f(x)\nu(dx)\,.
\]
Also, we denote by $\MM(E)$ the space of point measure over $E$, defined as
\[
\MM(E) := \left\{\sum_{i=1}^N \delta_{x_i} \,:\, x_i \in E\,,\, N \in \mathbb{N} \right \}\,,
\]
equipped with, for $f : E \to \RR$ and a measure $\nu \in \MM(E)$,
\[
\norm{f,\nu}_E:= \sum_{i=1}^N f(x_i)\,.
\]

In general, in the following, we will often consider either $E=\PP$ or $E=\QQ$; if no confusion is possible we will omit the subscript in the scalar product.

Fix a finite time horizon $T<\infty$, for $t \in [0,T]$, we define the concentration measure of lesion at time $t$, as
\begin{equation}\label{DEF:Nu1}
\nu(t) := \sum_{i=1}^{N(t)} \delta_{P_i(t)} = \sum_{i=1}^{N(t)} \delta_{Q_i(t)}\delta_{s_i}\,,
\end{equation}
with
\[
N(t) = \norm{\unit,\nu(t)}\,,
\]
the total number of lesions at time $t$. We further denote by $\nu^\XX(t)$ and $\nu^\YY(t)$ the marginal distributions
\begin{equation}\label{DEF:Nu2}
\nu^\XX(t)(\cdot) := \nu^\XX(t)\left(\, \cdot\,,\XX\right)\,,\quad \nu^\YY(t)(\cdot) := \nu^\YY(t)\left(\, \cdot\,,\YY\right)\,.
\end{equation}
Analogously to the previous notation, $N^\XX(t)$, resp. $N^\YY(t)$, denote the total number of lesions of type $\XX$, resp. $\YY$, at time $t$.

Besides lesion concentration we will often use a vector listing all lesions in the system; thus, given a system state $\nu(t)$, we denote by
\[
\HH(\nu(t)) := \left (\left(q^{1;\XX}(t),\XX\right),\dots,\left(q^{N^\XX(t);\XX}(t),\XX\right),\left(q^{1;\YY}(t),\YY\right),\dots,\left(q^{N^\YY(t);\YY}(t),\YY\right),0,\dots \right)\,,
\]
the position and type of all lesions in the system at time $t$. It is worth stressing that, since lesions of the same type are indistinguishable, the chosen ordering is arbitrary and there is no ambiguity in $\HH(\nu(t))$. We denote for short by $\HH^i(\nu(t)) \in \PP$, the $i-$th entry of the vector $\HH(\nu(t))$. With a similar notation, we denote
\[
\begin{split}
\HH(\nu^{\XX}(t)) &:= \left (q^{1;\XX}(t),\dots,q^{N^\XX(t);\XX}(t),0,\dots \right)\,,\\
\HH(\nu^{\YY}(t)) &:= \left (q^{1;\YY}(t),\dots,q^{N^\YY(t);\YY}(t),0,\dots \right)\,.
\end{split}
\]
the vector containing only the positions of lesions of type $\XX$ and $\YY$ respectively.

\subsection{The model}

Each lesion $i$, characterized by its position and lesion type $ P_i = \left (q_i,s_i\right)$, can move and undergo three different pathways. Such rates can be described by the system
\begin{equation}\label{EQN:Rates}
\begin{cases}
& \XX \xrightarrow{\RG} \emptyset\,,\\
& \XX \xrightarrow{\AG} Y\,,\\
& \XX + \XX \xrightarrow{\BG} 
\begin{cases}
\YY & \mbox{with probability} \quad \PG \in [0,1]\,,\\
\emptyset & \mbox{with probability} \quad 1-\PG \in [0,1]\,,\\
\end{cases}\,.\\
\end{cases}
\end{equation}
and can be characterized as follows:
\begin{description}[style=unboxed,leftmargin=0cm]
\item[(i) - Repair] each lesion in the class of sub-lethal lesions $\XX$ can repair at a rate 
\[
\RG :\QQ \times \RR \to \RR_+\,,\quad \RG\left(q,\norm{\Gamma^{\RG}_q,\nu}\right)\,,
\]
that depends on the spatial position of the $i-$th lesion and on the concentration of the system. A sub-lethal lesion that repairs disappear from the system.

The \textit{repair rate} $\RG$ is associated to a Poisson point measure 
\[
\PN^\RG(ds,d\ii,d\theta) \quad \mbox{on} \quad \RR_+ \times \NN_0 \times \RR_+\,.
\]
The index $\ii \in \NN_0$ gives the sampled lesion in $\XX$ to repair. The corresponding intensity measure associated with $N^\RG$ is
\[
\lambda^\RG (ds,d\ii,d\theta) := ds \otimes \left (\sum_{k \geq 0} \delta_k(\ii)\right) \otimes d\theta\,.
\]

We denote with $\PNC^\RG$ the compensated Poisson measure defined as
\[
\PNC^\RG(ds,d\ii,d\theta) := \PN(ds,d\ii,d\theta) - \lambda^\RG (ds,d\ii,d\theta)\,.
\]

\item[(ii) - Death] each lesion in the class of sub-lethal lesions $\XX$ can die at a rate 
\[
\AG :\QQ \times \RR \to \RR_+\,,\quad \AG\left(q,\norm{\Gamma^\AG_q,\nu}\right)\,,
\]
that depends on the spatial position of the $i-$th lesion and on the concentration of the system. A sub-lethal lesion that dies generates a lethal lesion $\YY$ at a new position $q$ according to the probability distribution $m^\AG(q)$, $q \in \QQ$.

The \textit{death rate} $\AG$ is associated to a Poisson point measure 
\[
\PN^\AG(ds,d\ii,d\theta_1,d\theta_2)\quad \mbox{on} \quad \RR_+ \times \NN_0 \times \RR_+ \times \RR_+\,.
\]
The index $\ii \in \NN_0$ gives the sampled lesion in $\XX$ to die and become a lethal lesion in $\YY$ in position $q$ sampled from $m^\AG(q|q_1)$, $q \in \QQ$. The corresponding intensity measure associated with $N^\AG$ is
\[
\lambda^\AG (ds,d\ii,d\theta_1,d\theta_2) := ds \otimes \left (\sum_{k \geq 0} \delta_k(\ii)\right)\otimes d\theta_1 \otimes d\theta_2\,.
\]

We denote with $\PNC^\AG$ the compensated Poisson measure defined as
\[
\PNC^\AG(ds,d\ii,d\theta_1,d\theta_2) := \PN^\AG(ds,d\ii,d\theta_1,d\theta_2) - \lambda^\AG (ds,d\ii,d\theta_1,d\theta_2)\,.
\]


\item[(iii) - Pairwise interaction] two lesions in the class of sub-lethal lesions $\XX$ can interact at a rate 
\[
\BG :\QQ \times \QQ \times \RR \to \RR_+\,,\quad \BG\left(q_1,q_2,\norm{\Gamma^\BG_{q_1,q_2},\nu}\right)\,,
\]
that depends on the spatial position of the $(i_1, i_2)-$th lesions and on the concentration of the system. Two sub-lethal lesions that interact can either (i) die with probability $p$, generating a lethal lesion $\YY$ at a new position $q$ according to the distribution $m^\BG(q)$, $q \in \QQ$, or (ii) repair with probability $1-p$ and disappear from the system. The probability $p$ depends also on the positions of the sampled lesions, namely
\[
\PG :\QQ \times \QQ \to [0,1]\,,\quad \PG\left(q_1,q_2\right)\,.
\]


The \textit{pairwise interaction rate} $\BG$ is associated to two Poisson point measure 
\[
\begin{split}
\PN^{\BG;p}(ds,d\ii,dq,d\theta_1,d\theta_2)\quad &\mbox{on} \quad \RR_+ \times \NN_0 \times \NN_0 \times \QQ \times \RR_+ \times \RR_+\,,\\
\PN^{\BG;1-p}(ds,d\ii,d\theta)\quad &\mbox{on} \quad \RR_+ \times \NN_0 \times \NN_0 \times \RR_+\,.
\end{split}
\]
The index $\ii = (\ii_1,\ii_2) \in \NN_0 \times \NN_0$ gives the sampled lesions in $\XX$ to either become a lethal lesion in $\YY$ in position $q$ sampled from $m(q|q_1,q_2)$, $q \in \QQ$, or repair and be removed from the system. The corresponding intensity measures associated with $\PN^{\BG;p}$ and $\PN^{\BG;1-p}$ are
\[
\begin{split}
\lambda^{\BG;p} (ds,d\ii,dq,d\theta_1,d\theta_2) &:= ds \otimes \left (\sum_{k \geq 0} \delta_k(\ii_1) \wedge \sum_{k \geq 0} \delta_k(\ii_2) \right) \otimes dq \otimes d\theta_1 \otimes d\theta_2\,,\\
\lambda^{\BG;1-p} (ds,d\ii,d\theta) &:= ds \otimes \left (\sum_{k \geq 0} \delta_k(\ii_1) \wedge \sum_{k \geq 0} \delta_k(\ii_2) \right) \otimes d\theta\,.
\end{split}
\]

We denote with $\PNC^\BG$ the compensated Poisson measure defined as
\[
\begin{split}
\PNC^{\BG;p}(ds,d\ii,dq,d\theta_1,d\theta_2) &:= \PN^{\BG;p}(ds,d\ii,dq,d\theta_1,d\theta_2) - \lambda^{\BG;p} (ds,d\ii,dq,d\theta_1,d\theta_2)\,,\\
\PNC^{\BG;1-p}(ds,d\ii,d\theta) &:= \PN^{\BG;1-p}(ds,d\ii,d\theta) - \lambda^{\BG;1-p} (ds,d\ii,d\theta)\,.\\
\end{split}
\]

\item[(iv) - Spatial diffusion] each lesion of type $\XX$ and $\YY$ moves around the domain $\QQ$ with diffusion term 
\[
\begin{split}
&\sigma^\XX: \QQ \to \RR^{d \times d} \,,\quad \sigma^\XX\left(q\right)\,,\\
&\sigma^\YY: \QQ \to \RR^{d \times d} \,,\quad \sigma^\YY\left(q\right)\,,\\
\end{split}
\] 
and drift term 
\[
\begin{split}
&\mu^\XX: \QQ \to \RR^d \,,\quad \mu^\XX\left(q\right)\,,\\
&\mu^\YY: \QQ \to \RR^d\,,\quad \mu^\YY\left(q\right)\,.
\end{split}
\]
In the following, we will also denote
\[
\begin{split}
&\Sigma^\XX: \QQ \to \mathcal{S}_+\left(\RR^d \right) \,,\quad \Sigma^\XX\left(q\right) := \sigma^\XX\left(q\right)\left(\sigma^\XX\left(q\right)\right)^T\,,\\
&\Sigma^\YY: \QQ \to \mathcal{S}_+\left( \RR^d \right) \,,\quad \Sigma^\YY\left(q\right) := \sigma^\YY\left(q\right) \left (\sigma^\YY\left(q\right)\right)^T\,,\\
\end{split}
\]
with $ \mathcal{S}_+\left(\RR^d \right)$ the space of symmetric non-negative $d \times d$ matrices.

To describe lesion motion we introduce a countable collection of standard independent Brownian motion $\left (W^{n;\XX}(t)\right )_{n \in \NN}$ and $\left (W^{n;\YY}(t)\right )_{n \in \NN}$ on $\RR^d$. Brownian motion is assumed to reflect with normal derivative at the boundary of the domain $\QQ$. In particular, denote by $T_{n}$ and $T_{n+1}$ two successive jump times of the process $\nu$, and assume that at time $T_n$ we have $N^\XX(T_n)$, resp. $N^\YY(T_n)$, lesions of type $\XX$, resp. $\YY$. It is worth stressing that in $t \in [T_{n},T_{n+1})$, the number of lesions remains constant so that the process is solely subject to the diffusive component. Thus, for any $t \in [T_{n},T_{n+1})$ each lesion evolves according to the following SDE with reflection at the boundaries
\begin{equation}\label{EQN:BrownRef}
\setlength{\jot}{100pt}
\begin{cases}
\XX^{i_\XX}(t) &= \XX^{i_\XX}(T_n) + \int_{T_n}^{t} \sigma^\XX \left(\XX^{i_\XX}(s)\right) dW^{i_\XX;\XX}(s)+ \int_{T_n}^{t} \mu^\XX \left (\XX^{i_\XX}(s) \right) ds - \kappa^{i_\XX}(t)\,,\\[2pt]
|\kappa^{i_\XX}|(t) &= \int_{T_n}^{t} \Ind{\XX^{i_\XX}(s) \in \partial \QQ} d|\kappa^{i_\XX}|(s)\,,\quad \kappa^{i_\XX}(t) = \int_{T_n}^{t} n\left ( \XX^{i_\XX}(s) \right)d|\kappa^{i_\XX}|(s)\,,\\[2pt]
\XX^{i_\XX}(t) &\in \bar{\QQ}\,,\quad i_\XX+1,\dots,N^\XX(t)\,,\\[4pt]
\YY^{i_\YY}(t) &= \YY^{i_\YY}(T_n) + \int_{T_n}^{t} \sigma^\YY \left (\YY^{i_\YY}(s)\right) dW^{i_\YY;\YY}(s)+ \int_{T_n}^{t} \mu^\YY \left (\YY^{i_\YY}(s)\right) ds - \kappa^{i_\YY}(t)\,,\\[2pt]
|\kappa^{i_\YY}|(t) &= \int_{T_n}^{t} \Ind{\YY^{i_\YY}(s) \in \partial \QQ} d|\kappa^{i_\YY}|(s)\,,\quad \kappa^{i_\YY}(t) = \int_{T_n}^{t} n\left ( \YY^{i_\YY}(s) \right)d|\kappa^{i_\YY}|(s)\,,\\[2pt]
\YY^{i_\Y}(t) &\in \bar{\QQ}\,,\quad i_\YY+1,\dots,N^\YY(t)\,,\\[2pt]
\end{cases}
\end{equation}
where we denoted by $dW(t)$ the integration in the sense of It\^o.
\end{description}

In the following, we will consider a filtered and complete probability space $\left (\Omega,\mathcal{F},\left (\mathcal{F}_t\right )_{t \in \RR_+},\mathbb{P}\right )$ satisfying standard assumptions, namely right--continuity and saturation by $\mathbb{P}$--null sets. In particular, $\left (\mathcal{F}_t\right )_{t \in \RR_+}$ is the filtration generated by the processes defined in $(i)-(ii)-(iii)-(iv)$ as well as a $\MM \times \MM-$valued initial distribution $\nu = (\nu^\XX_0,\nu^\YY_0)$.

\begin{Remark}
Notice that, compared to the original interaction rates as introduced in \cite{cordoni2021generalized}, we included in the present version of the model a further possible pathway, namely
\[
\XX + \XX \xrightarrow{\BG} 
\begin{cases}
\YY & \mbox{with probability} \quad \PG \in [0,1]\,,\\
\emptyset & \mbox{with probability} \quad 1-\PG \in [0,1]\,,\\
\end{cases}\,.\\
\]
This is done since, as noted in early versions of advanced radiobiological models \cite{sachs1992dna}, pairwise interaction of damages, can result also in correct repairs; such a process is called for instance in \cite{sachs1992dna} as \textit{complete exchange}.
\demo
\end{Remark}

Through the paper we will assume the following hypothesis to hold:

\begin{Hypothesis}\label{HYP:1}
\begin{enumerate}
    \item Jump components:
    \begin{description}
            \item[(2.i)] the \textit{repair rate} $\RG$ is uniformly bounded over compact subsets, that is for $N \geq 0$ it exists $\bar{\RG}$ such that
            \[
            \sup_{q \in \QQ} \sup_{v \in [0,N]} \RG(q,v) < \bar{\RG} < \infty\,;
            \]
            \item[(2.ii)] the \textit{death rate} $\AG$ satisfies a linear growth condition, that is, there exists a positive constant $\bar{\AG}$ such that, for all $q \in \QQ$ it holds
            \[
            \begin{split}
                &0 \leq \AG(q,v) \leq \bar{\AG}(1+|v|)\,;\\
            \end{split}
            \]
            \item[(2.iii)] the \textit{pairwise interaction rate} $\BG$ satisfies a linear growth condition, that is, there exists a positive constant $\bar{\BG}$ such that, for all $q_1$ and $q_2 \in \QQ$, it holds
            \[
            \begin{split}
                &0 \leq \BG(q_1,q_2,v) \leq \bar{\BG}(1+|v|)\,;\\
            \end{split}
            \]
            \item[(2.iv)] the \textit{pairwise interaction death} $\PG$ is a probability, that is, for all $q_1$ and $q_2 \in \QQ$, it holds
            \[
            \begin{split}
                 &\PG(q_1,q_2) \in [0,1]\,;\\
            \end{split}
            \]
        \end{description}
    \item Diffusive components:
        \begin{description}
            \item[(2.i)] there exist positive constants $L^\XX$ and $L^\YY$ such that, for any $q_1$, $q_2 \in \QQ$, it holds
            \[
            \begin{split}
                &|\sigma^\XX(q_1)-\sigma^\XX(q_2)|+|\mu^\XX(q_1)-\mu^\XX(q_2)|\leq L^\XX |q_1 - q_2| \,,\\
                &|\sigma^\YY(q_1)-\sigma^\YY(q_2)|+|\mu^\YY(q_1)-\mu^\YY(q_2)|\leq L^\YY |q_1 - q_2| \,.
            \end{split}
            \]
        \end{description}
         \item Kernel components:
        \begin{description}
            \item[(3.i)] for all $h \in \{a,b,r,\sigma^\XX,\mu^\XX,\sigma^\YY,\mu^\YY,\}$ and $Q \in \QQ$ the function $\Gamma^h_Q : \QQ \times \Sp \to \RR_+$ is continuous and uniformly bounded, that is, there exists a constant $\bar{\Gamma}^h$ such that
            \[
            \sup_{q \in \QQ}\, \sup_{(\bar{q},\bar{s}) \in \QQ \times \Sp} \Gamma^h_q (\bar{q},\bar{s}) < \bar{\Gamma}^h<\infty\,;
            \]
            \item[(3.ii)] the sampling measure $m(q|q_1,q_2)$ is a probability measure, that is
            \[
            \int_\QQ m(q|q_1,q_2)dq =1\,.
            \]
        \end{description}
\end{enumerate}
\end{Hypothesis}

In the following we consider the next class of cylindrical test functions: for $F \in \CC^2_b (\RR \times \RR)$, that is, $F$ is bounded with continuous and bounded second order derivative, and for $f^\XX\,,\,f^\YY \in \CC^{2}_0(\QQ)$, that is $f^\XX$ and $f^\YY$ are continuous with bounded second order derivative in the domain variable $\QQ$, satisfying $\nabla_q f(q) \cdot n(q)=0$, we consider $F_{(f^\XX,f^\YY)}:\MM \times \MM \to \RR$ of the form
\begin{equation}\label{EQN:Cyl}
F_{(f^\XX,f^\YY)}(\nu) = F \left (\norm{f^\XX,\nu},\norm{f^\YY,\nu}\right)\,.
\end{equation}
In the following, we will denote by $\frac{\partial}{\partial x}$, resp. $\frac{\partial}{\partial x}$, the derivative with respect to the first argument, resp. the second argument, of the function $F$. Also, $\nabla$, resp. $\Delta$, resp. $Tr$, resp. $Hess$, denotes the gradient with respect to the space variable $q$, resp. the Laplacian operator with respect to the space variable $q$, resp. the trace operator, resp. the Hessian matrix. Cylindrical functions \eqref{EQN:Cyl} are a standard class generating the set of bounded and measurable functions from $\MM \times \MM$ into $\RR$. 

\begin{Remark}
\begin{description}
    \item[(i)] A natural choice for the kernel $\Gamma$ would be to assume that only nearby mass affects the overall rate; in such a case we have, for $q \in \QQ$,
    \[
    \Gamma_q (\bar{q},\bar{s}) := \Ind{|q-\bar{q}|<\epsilon} (\bar{q},\bar{s})\,.
    \]
    Therefore, only lesions that are at most distant $\epsilon$ from the position $q$ where the reaction happens to participate in the reaction.
    
    \item[(ii)] Regarding the \textit{pairwise interaction rate} $\BG$, it is natural to assume that $\BG$ depends only on the separation distance between two lesions, that is, it exists a function
    \[
    \bar{\BG} :\RR \to \RR_+\,,\quad \bar{\BG}\left(\mathfrak{q}\right)\,,
    \]
    such that
    \[
    \BG\left(q_1,q_2\right) = \BG\left(q_2,q_1\right) = \bar{\BG}\left(|q_1 - q_2|\right)\,.
    \]
    Further, it is natural to assume that the closer two lesions are, the more likely to interact they are. Therefore, relevant choices for the rate $\BG$ are for instance a step interaction rate
    \begin{equation}\label{EQN:KerStep}
    \bar{\BG}\left(\mathfrak{q}\right) := \hat{\BG} \Ind{|\mathfrak{q}|<\varepsilon}\,,
    \end{equation}
    or a Gaussian rate
    \begin{equation}\label{EQN:KerGauss}
    \bar{\BG}\left(\mathfrak{q}\right) := \frac{\hat{\BG}}{\sqrt{2 \pi \varepsilon^2}}e^{-\frac{|\mathfrak{q}|^2}{2\varepsilon^2}} \,.
     \end{equation}
    
    Whereas the former rate \eqref{EQN:KerStep} models the case where only lesion closer than $\varepsilon$ can interact, the latter rate \eqref{EQN:KerGauss} considers that the rate of interaction decreases exponentially as the lesions are more distant from each other. At last, as noted in \cite{kellerer1974theory}, enhanced short-range interaction can be modelled using
    \begin{equation}\label{EQN:KerGaussEn}
    \bar{\BG}\left(\mathfrak{q}\right) := \frac{\hat{\BG}_1}{\sqrt{2 \pi \varepsilon^2_1}}e^{-\frac{|\mathfrak{q}|^2}{2\varepsilon^2_1}} + \frac{\hat{\BG}_2}{\sqrt{2 \pi \varepsilon^2_2}}e^{-\frac{|\mathfrak{q}|^2}{2\varepsilon^2_2}} \,,
     \end{equation}
     for suitable constants, so that the rate of interaction declines fast but still has a fat tail at larger distances. Similarly, it is reasonable to assume that also $\PG$ depends on the distance between the interacting lesions.
     
    \item[(iii)] Possible choices for the sampling measure $m(q|q_1,q_2)$ would be to assume that, whenever two sub-lethal lesions at $(q_1,q_2)$ reacts, a lethal lesion is created randomly on the segment connecting $q_1$ and $q_2$. For instance, the following probability distributions can be considered, 
    \[
    m(q|q_1,q_2) = \sum_{j=1}^J p_j \delta_{\alpha_j q_1 + (1-\alpha_j)q_2}(q)\,,\quad \alpha_j \in [0,1]\quad \mbox{and} \quad \sum_{j=1}^J p_j=1\,,
    \]
    or also
    \[
    m(q|q_1,q_2) =\frac{1}{|q_1 - q_2|} \delta_{\alpha q_1 + (1-\alpha)q_2}(q)\,,\quad \alpha \in [0,1]\,.
    \]
\end{description}
\demo
\end{Remark}

With the previous notation, we can thus introduce the following weak representation for the spatial radiation-induced DNA lesion repair model, given $f^\XX$ and $f^\YY \in C^2_0(\RR)$, we have
{\footnotesize
\begin{equation}\label{EQN:SpatGSM2}
\begin{cases}
\norm{f^\XX,\nu^\XX(t)} &= \norm{f^\XX,\nu^\XX(0)} + \int_0^t \sum_{i=1}^{N^\XX (s_-)} \sigma^\XX \left( \HH^i(\nu^\XX)\right) \cdot \nabla f^\XX\left( \HH^i(\nu^\XX)\right)d W^{i:\XX}(s) + \\[2pt]
&+\int_0^t \sum_{i=1}^{N^\XX (s_-)} \mu^\XX \left( \HH^i(\nu^\XX)\right) \cdot \nabla f^\XX\left( \HH^i(\nu^\XX)\right)ds +\\[3pt]
&+\frac{1}{2} \int_0^t \sum_{i=1}^{N^\XX (s_-)}Tr\left [\Sigma^\XX\left(\HH^i(\nu^\XX)\right)
 \mbox{Hess}\,  f^\XX\left( \HH^i(\nu^\XX)\right) \right]ds+\\[3pt]
&+ \int_0^t \int_{\NN_0} \int_{\RR_+} \left[ \norm{f^\XX,\nu^{\XX}(s_-) - \delta_{\HH^i\left( \nu^\XX(s_-)\right)}} - \norm{f^\XX,\nu^{\XX}(s_-)}\right]  \Ind{i \leq N^\XX (s_-)}\Ind{\theta \leq \RG\left( \HH^i\left( \nu^\XX(s_-)\right)\right)} \PN^\RG(ds,di,d\theta)+\\[3pt]
&+ \int_0^t \int_{\NN_0} \int_{\RR_+^2} \left[ \norm{f^\XX,\nu^{\XX}(s_-) - \delta_{\HH^i\left( \nu^\XX(s_-)\right)}} - \norm{f^\XX,\nu^{\XX}(s_-)}\right]\Ind{i \leq N^\XX (s_-)}\Ind{\theta_1 \leq \AG\left( \HH^i\left( \nu^\XX(s_-)\right)\right)}  \PN^\AG(ds,di,d\theta_1,d\theta_2)+\\[3pt]
&+ \int_0^t \int_{\NN_0^2} \int_{\QQ} \int_{\RR_+^2} \left[ \norm{f^\XX,\nu^{\XX}(s_-) - \delta_{\HH^{i_1}\left( \nu^\XX(s_-)\right)} - \delta_{\HH^{i_2}\left( \nu^\XX(s_-)\right)}} - \norm{f^\XX,\nu^{\XX}(s_-)}\right] \times \\[3pt]
&  \times \Ind{i_1 < i_2 \leq N^\XX (s_-)}\Ind{\theta_1 \leq p\left( \HH^{i_1}\left( \nu^\XX(s_-)\right),\HH^{i_2}\left( \nu^\XX(s_-)\right)\right)\BG\left( \HH^{i_1}\left( \nu^\XX(s_-)\right),\HH^{i_2}\left( \nu^\XX(s_-)\right)\right)} \Ind{\theta_2 \leq m(q|\HH^{i_1}\left( \nu^\XX(s_-)\right),\HH^{i_2}\left( \nu^\XX(s_-)\right))} \times \\[3pt]
& \qquad \qquad \qquad \qquad  \times \PN^{\BG;p}(ds,di_1,di_2,dq,d\theta_1,d\theta_2)\,,\\[3pt]
&+ \int_0^t \int_{\NN_0^2}\int_{\RR_+} \left[ \norm{f^\XX,\nu^{\XX}(s_-) - \delta_{\HH^{i_1}\left( \nu^\XX(s_-)\right)} - \delta_{\HH^{i_2}\left( \nu^\XX(s_-)\right)}} - \norm{f^\XX,\nu^{\XX}(s_-)}\right] \times \\[3pt]
&  \times \Ind{i_1 < i_2 \leq N^\XX (s_-)}\Ind{\theta \leq \left (1-p\left( \HH^{i_1}\left( \nu^\XX(s_-)\right),\HH^{i_2}\left( \nu^\XX(s_-)\right)\right)\right)\BG\left( \HH^{i_1}\left( \nu^\XX(s_-)\right),\HH^{i_2}\left( \nu^\XX(s_-)\right)\right)}\PN^{\BG;1-p}(ds,di_1,di_2,d\theta)\,,\\[5pt]
\norm{f^\YY,\nu^\YY(t)} &= \norm{f^\YY,\nu^\YY(0)} + \int_0^t \sum_{i=1}^{N^\YY (s_-)} \sigma^\YY \left( \HH^i(\nu^\YY)\right) \cdot \nabla f^\YY\left( \HH^i(\nu^\YY)\right)d W^{i:\YY}(s) + \\[3pt]
&+\int_0^t \sum_{i=1}^{N^\YY (s_-)} \mu^\YY \left( \HH^i(\nu^\YY)\right) \cdot \nabla f^\YY\left( \HH^i(\nu^\YY)\right)ds +\\[3pt]
&+\frac{1}{2} \int_0^t \sum_{i=1}^{N^\YY (s_-)}Tr\left [\Sigma^\YY\left(\HH^i(\nu^\YY)\right)
 \mbox{Hess}\,  f^\YY\left( \HH^i(\nu^\YY)\right) \right]ds+\\[3pt]
&+ \int_0^t \int_{\NN_0} \int_{\RR_+^2} \left[ \norm{f^\YY,\nu^{\YY}(s_-) + \delta_{\HH^i\left(\nu^\XX(s_-)\right)}} - \norm{f^\YY,\nu^{\YY}(s_-)}\right] \Ind{i \leq N^\XX (s_-)}\Ind{\theta_1 \leq \AG\left( \HH^i\left( \nu^\XX(s_-)\right)\right)}  \PN^\AG(ds,di,d\theta_1,d\theta_2)+\\[3pt]
&+ \int_0^t \int_{\NN_0^2} \int_{\QQ} \int_{\RR_+^2} \left[ \norm{f^\YY,\nu^{\YY}(s_-) + \delta_{q}} - \norm{f^\YY,\nu^{\YY}(s_-)}\right] \times \\[3pt]
&\times \Ind{i_1 < i_2 \leq N^\XX (s_-)}\Ind{\theta_1 \leq p\left( \HH^{i_1}\left( \nu^\XX(s_-)\right),\HH^{i_2}\left( \nu^\XX(s_-)\right)\right)\BG\left( \HH^{i_1}\left( \nu^\XX(s_-)\right),\HH^{i_2}\left( \nu^\XX(s_-)\right)\right)} \Ind{\theta_2 \leq m(q|\HH^{i_1}\left( \nu^\XX(s_-)\right),\HH^{i_2}\left( \nu^\XX(s_-)\right))} \times \\[3pt]
& \qquad \qquad \qquad \qquad  \times \PN^{\BG;p}(ds,di_1,di_2,dq,d\theta_1,d\theta_2)\,.
\end{cases}
\end{equation}
}

\begin{Definition}\label{DEF:DefNu}
We say that $\nu(t) = (\nu^\XX(t),\nu^\YY(t))$ as defined in equations \ref{DEF:Nu1}--\ref{DEF:Nu2} is a \textit{spatial radiation-induced DNA damage repair model} if $\nu = \left(\nu(t)\right)_{t \in \RR_+}$ is $\left (\mathcal{F}_t\right )_{t \in \RR_+}-$adapted and for any $f^\XX$ and $f^\YY\in C^2_0(\QQ)$ equation \eqref{EQN:SpatGSM2} holds $\mathbb{P}-$a.s.
\end{Definition}

The above process is characterized by the following infinitesimal generator
\begin{equation}\label{EQN:InfGenAll}
\LL F_{(f^\XX,f^\YY)}(\nu) = \LL_d F_{(f^\XX,f^\YY)}(\nu) + \sum_{h \in \{\RG,\AG,\BG\}} \LL_{h} F_{(f^\XX,f^\YY)}(\nu)\,,    
\end{equation}
where $\LL F_{(f^\XX,f^\YY)}(\nu)$ is the infinitesimal generator of the reaction terms, whereas $\LL_d F_{(f^\XX,f^\YY)}(\nu)$ is the infinitesimal generator of the diffusive part of equation \eqref{EQN:SpatGSM2}.

In particular, we have
\begin{equation}\label{EQN:InfGD}
\begin{split}
\LL_d^\XX f^\XX(q) &= \mu^\XX \left(q\right) \cdot \nabla f^\XX\left(q\right)ds +\frac{1}{2} Tr\left [\Sigma^\XX\left(q\right) \mbox{Hess} \,f^\XX\left(q\right)\right ]\,,\\
\LL_d^\XX f^\YY(q) &= \mu^\YY \left(q\right) \cdot \nabla f^\YY\left(q\right)ds +\frac{1}{2} Tr\left [\Sigma^\YY\left(q\right) \mbox{Hess}\, f^\YY\left(q\right)\right ]\,.
\end{split}
\end{equation}

It further holds that
\begin{equation}\label{EQN:InfDiff}
\begin{split}
\LL_d F_{(f^\XX,f^\YY)}(\nu) &= \norm{\LL_d^\XX f^\XX,\nu^\XX} \frac{\partial}{\partial x}  F \left (\norm{f^\XX,\nu},\norm{f^\YY,\nu}\right) + \norm{\LL_d^\YY f^\YY,\nu^\YY} \frac{\partial}{\partial y}  F \left (\norm{f^\XX,\nu},\norm{f^\YY,\nu}\right) + \\
&+ \norm{\nabla^T f^\XX\, \sigma^\XX \,\nabla f^\XX,\nu^\XX} \frac{\partial^2}{\partial x^2}  F \left (\norm{f^\XX,\nu},\norm{f^\YY,\nu}\right) + \norm{\nabla^T f^\YY \, \sigma^\YY \,\nabla f^\YY,\nu^\YY} \frac{\partial^2}{\partial y^2}  F \left (\norm{f^\XX,\nu},\norm{f^\YY,\nu}\right)\,.
\end{split}
\end{equation}

Regarding the infinitesimal generator of the reaction terms it holds
\begin{equation}\label{EQN:InfReact}
\begin{split}
\LL_{\RG} F_{(f^\XX,f^\YY)}(\nu) &= \int_{\QQ} \RG(q,\nu)\left[F_{(f^\XX,f^\YY)}(\nu^\XX - \delta_{q},\nu^\YY) - F_{(f^\XX,f^\YY)}(\nu) \right]\nu^\XX(dq)\,,\\
\LL_{\AG} F_{(f^\XX,f^\YY)}(\nu) &= \int_{\QQ} \int_{\QQ} \AG(q,\nu)\left[F_{(f^\XX,f^\YY)}((\nu^\XX - \delta_{q},\nu^\YY + \delta_{\bar{q}})) - F_{(f^\XX,f^\YY)}(\nu) \right]m^{\AG}(\bar{q}|q)d\bar{q}\nu^\XX(dq)\,,\\
\LL_{\BG} F_{(f^\XX,f^\YY)}(\nu) &= \int_{\tilde{\QQ}^2} \int_{\QQ} \PG(q_1,q_2)\BG(q_1,q_2,\nu)\left[F_{(f^\XX,f^\YY)}((\nu^\XX - \delta_{q_1} - \delta_{q_2},\nu^\YY + \delta_{\bar{q}})) - F_{(f^\XX,f^\YY)}(\nu) \right] \times \\
&\qquad \qquad \qquad \qquad \times m^{\BG}(\bar{q}|q_1,q_2)d\bar{q}\nu^\XX(dq_1)\nu^\XX(dq_2) + \\
&+\int_{\tilde{\QQ}^2} \left (1-\PG(q_1,q_2)\right) \BG(q_1,q_2,\nu)\left[F_{(f^\XX,f^\YY)}((\nu^\XX - \delta_{q_1} - \delta_{q_2},\nu^\YY)) - F_{(f^\XX,f^\YY)}(\nu) \right] \nu^\XX(dq_1)\nu^\XX(dq_2)\,,\\
\end{split}
\end{equation}
where we have denoted by
\[
\tilde{\QQ}^2 := \QQ^2 \setminus \left \{(q_1,q_2) \,:\, q_1=q_2 \right \}\,.
\]

\subsection{Stepwise construction of the process}

In the present Section, we provide a step-wise construction of the process. Such construction, besides being relevant from a theoretical point of view, is particularly important in implementing a simulation algorithm for the process defined in the previous section. Notice that, using assumptions \ref{HYP:1}, we have that the rate at which $a$, $b$ and $r$ happen is bounded in the uniform norm. Thus, between the occurrence times of the jump components, each lesion moves according to the diffusive generator $\DD_X$ and $\DD_y$.

\begin{enumerate}
    \item starts with a random measure
    \[
    \nu_0 := \left (\nu^\XX_0 , \nu^\YY_0 \right):= \left (\sum_{i=1}^{N^\XX_0} \delta_{\XX^i(0)},\sum_{i=1}^{N^\YY_0} \delta_{\YY^i(0)}\right)\,,
    \]
    and set $t=\tau^0 = 0$. The initial distribution $M_0$ will be treated explicitly and in detail in Section \ref{SEC:Init};\\
    \item every jump reaction $h \in \{a,b,r\}$ has an exponential clock; set thus the random time of the first reaction happening
    \[
    \begin{split}
    \tau_h^1 := \inf \left \{t > 0 \, : \, \int_0^t \bar{\mathrm{h}}(\nu(s))ds \geq \mathcal{E}^1_h \right \}\,,\quad \bar{h}\in \left \{\bar{\RG},\bar{\AG},\bar{\BG}\right\}
    \end{split}
    \]
    with $\mathcal{E}^1_h$ is an exponential random variable with parameter $1$. Also, we have defined
    \[
    \begin{split}
    \bar{\mathrm{r}}(\nu(s)) &:= \sum_{i=1}^{N^\XX(s)} \RG\left(\HH^i\left (\nu^\XX(s)\right),\norm{\Gamma^{\RG}_{\HH^i\left (\nu^\XX(s)\right)},\nu}\right)\,,\\
    \bar{\mathrm{a}}(\nu(s)) &:= \sum_{i=1}^{N^\XX(s)} \AG\left(\HH^i\left (\nu^\XX(s)\right),\norm{\Gamma^{\AG}_{\HH^i\left (\nu^\XX(s)\right)},\nu}\right)\,,\\
    \bar{\mathrm{b}}(\nu(s)) &:= \sum_{\substack{i_1,i_2=1\\ i_1< i_2}}^{N^\XX(s)} \BG\left(\HH^{i_1}\left (\nu^\XX(s)\right),\HH^{i_2}\left (\nu^\XX(s)\right),\norm{\Gamma^{\BG}_{\HH^{i_1}\left (\nu^\XX(s)\right),\HH^{i_2}\left (\nu^\XX(s)\right)},\nu}\right)\,.\\
    \end{split}
    \]
    \item set $\tau^1 := \min_{h \in \{a,b,r\}}\tau_h^1$ and consider $h^1 \in \{a,b,r\}$ the reaction that triggers the random time $\tau^1_{h^1}$;\\
    \item let any lesion move according to the diffusion and drift coefficients as described in equation \eqref{EQN:BrownRef}, until time $T \wedge \tau_1$ is reached. If $T$ is reached exit, otherwise go to the next step;\\
    \item sample the lesions and positions of the lesions that triggered the reaction and, in case either $\AG$ or $\BG$ fired, sample the position $q$ of the new lesion $\YY$ created;
    \item at time $\tau^1$, if:
    \begin{description}
        \item[6.i] $\RG$ has been triggered, set $N^\XX(\tau^1) = N^\XX(\tau^1_-) -1$ and $N^\YY(\tau^1) = N^\YY(\tau^1_-)$ and remove the $i-$th component of $\XX$, that is we have
        \[
        (\XX^1(\tau^1),\dots,\XX^{i-1}(\tau^1),\XX^{i+1}(\tau^1),\dots,\XX^{N^X}(\tau^1))\,;
        \]
        \item[6.ii] $\AG$ has been triggered, set $N^\XX(\tau^1) = N^\XX(\tau^1_-) - 1$ and $N^\YY(\tau^1) = N^\YY(\tau^1_-)+1$, remove the $i-$th component of $\XX$ and create a new lesion $\YY$ in the same position. We thus have
        \[
        \begin{split}
        &(\XX^1(\tau^1),\dots,\XX^{i-1}(\tau^1),\XX^{i+1}(\tau^1),\dots,\XX^{N^\XX}(\tau^1))\,,\\
        &(\YY^1(\tau^1),\dots,\YY^{N^\YY}(\tau^1_-),\YY^{N^\YY+1}(\tau^1))\,.
        \end{split}
        \]
        \item[6.iii] $\BG$ has been triggered, and simulate a number $\tilde{p}$ from a random variable $P \sim U(0,1)$, if:
        \begin{description}
        \item[6.iii.a] $\tilde{p} \leq \PG\left(\XX^{i_1}(\tau^1),\XX^{i_2}(\tau^1)\right)$, then set $N^\XX(\tau^1) = N^\XX(\tau^1_-) - 2$ and $N^\YY(\tau^1) = N^\YY(\tau^1_-)+1$, remove the $i_1-$th and $i_2-$th component of $\XX$ and create a new lesion $\YY$ in position $q \in \QQ$. We thus have
        \[
        \begin{split}
        &(\XX^1(\tau^1),\dots,\XX^{i_1-1}(\tau^1),\XX^{i_1+1}(\tau^1),\dots,\XX^{i_2-1}(\tau^1),\XX^{i_2+1}(\tau^1),\dots,\XX^{N^\XX}(\tau^1))\,,\\
        &(\YY^1(\tau^1),\dots,\YY^{N^\YY}(\tau^1_-),\YY^{N^\YY+1}(\tau^1))\,.
        \end{split}
        \]
        \item[6.iii.b] $\tilde{p} > \PG\left(\XX^{i_1}(\tau^1),\XX^{i_2}(\tau^1)\right)$, then set $N^\XX(\tau^1) = N^\XX(\tau^1_-) - 2$ and $N^\YY(\tau^1) = N^\YY(\tau^1_-)$, remove the $i_1-$th and $i_2-$th component of $\XX$ remove lesions $i_1$ and $i_2$ from the system. We thus have
        \[
        \begin{split}
        &(\XX^1(\tau^1),\dots,\XX^{i_1-1}(\tau^1),\XX^{i_1+1}(\tau^1),\dots,\XX^{i_2-1}(\tau^1),\XX^{i_2+1}(\tau^1),\dots,\XX^{N^\XX}(\tau^1))\,,\\
        &(\YY^1(\tau^1),\dots,\YY^{N^\YY}(\tau^1_-))\,.
        \end{split}
        \]
        \end{description}
    \end{description}
    \item update $t = t + \tau^1$;\\
    \item if $t<T$, go to step 1 and repeat until $T$ is reached.
\end{enumerate}

\subsection{Well-posedness and martingale properties}

We can in the present Section the existence and uniqueness of solutions to the above-introduced model.

\begin{Theorem}\label{THM:E!}
Let $\nu^\XX_0$ and $\nu^\YY_0$ two independent random measures with finite $p-$th moment, $p \geq 1$, that is it holds
\begin{equation}\label{EQN:InitNu}
\EE \norm{\mathbf{1},\nu^\XX_0}^p < \infty\,,\quad \EE \norm{\mathbf{1},\nu^\YY_0}^p < \infty\,.
\end{equation}
Then, under Hypothesis \ref{HYP:1}, for any $T>0$, there exists a pathwise unique strong solution to the system \eqref{EQN:SpatGSM2} in $\DD \left ([0,T],\MM \times \MM \right)$. Also, it holds
\begin{equation}\label{EQN:SupEst}
\mathbb{E} \sup_{t \leq T} \norm{\mathbf{1},\nu (t)}^p  < \infty\,.
\end{equation}

In particular, the process $\nu$ in Definition \ref{DEF:DefNu} is well-defined on $\RR_+$.
\end{Theorem}
\begin{proof}
Since the jump times are isolated, the construction of $\nu^\XX(t)$ and $\nu^\YY(t)$ can be done pathwise inductively along the successive jump times. In particular, denote by $T_{m}$ and $T_{m+1}$ two successive jump times of the process $\nu$, and assume that at time $T_m$ we have $N^\XX(T_m)$, resp. $N^\YY(T_m)$, lesions of type $\XX$, resp. $\YY$. As noted above, for $t \in [T_{m},T_{m+1})$ the number of lesions remains constant so that the process is solely subject to the diffusive component as described in equation \eqref{EQN:BrownRef}. Using Hypothesis \ref{HYP:1}, conditionally on $\mathcal{F}_{T_m}$, equation \eqref{EQN:BrownRef} can be seen as a purely diffusive SDE with globally Lipschitz coefficients on $\RR^{d \times N^\XX(T_m)} \times \RR^{d \times N^\YY(T_m)}$, so that the process 
\[
\left(\XX^{i_\XX}(t),\YY^{i_\YY}(t) \right)_{i_\XX=1,\dots,N^\XX(T_n);\, i_\YY=1,\dots,N^\YY(T_n)}\,,
\]
admits a unique strong solution for $t \in [T_{m},T_{m+1})$.

Define then, for $n \geq 0$,
\[
\tau_n^\XX := \inf \{t \geq 0 \,:\, \norm{\mathbf{1},\nu^\XX(t)} \geq n\}\,,
\]
and set for short $\bar{\tau}_n^\XX := t \wedge \tau_n^\XX$. 

We can construct a solution algorithmically in $[0,T)$. For $t \geq 0$, noticing that the number of lesions in $\XX$ can only decrease, we have,
\begin{equation}\label{EQN:Est1}
\begin{split}
\sup_{s \in [0,\bar{\tau}_n^\XX]} \norm{\mathbf{1},\nu^\XX(s)}^p & \leq \norm{\mathbf{1},\nu^\XX_0}^p \,.
\end{split}
\end{equation}

Regarding $\YY$, we have that,
\begin{equation}\label{EQN:Est2}
\begin{split}
&\sup_{s \in [0,\bar{\tau}_n^\YY]} \norm{\mathbf{1},\nu^\YY(s)}^p \leq \norm{\mathbf{1},\nu^\YY_0}^p +\\
&+\int_0^{\bar{\tau}_n^\YY} \int_{\NN_0} \int_{\RR_+^2} \left[ \left (\norm{\mathbf{1},\nu^{\YY}(s_-)} + 1\right)^p - \norm{\mathbf{1},\nu^{\YY}(s_-)}^p\right] \Ind{i \leq N^\XX (s_-)}\Ind{\theta_1 \leq \AG\left( \HH^i\left( \nu^\XX(s_-)\right)\right)}  \PN^\AG(ds,di,d\theta_1,d\theta_2)+\\[3pt]
&+ \int_0^{\bar{\tau}_n^\YY} \int_{\NN_0^2} \int_{\QQ} \int_{\RR_+^2} \left[ \left(\norm{\mathbf{1},\nu^{\YY}(s_-)} + 1\right)^p - \norm{\mathbf{1},\nu^{\YY}(s_-)}^p\right] \times \\[3pt]
&\times \Ind{i_1 < i_2 \leq N^\XX (s_-)}\Ind{\theta_1 \leq \BG\left( \HH^{i_1}\left( \nu^\XX(s_-)\right),\HH^{i_2}\left( \nu^\XX(s_-)\right)\right)} \Ind{\theta_2 \leq m(q|\HH^{i_1}\left( \nu^\XX(s_-)\right),\HH^{i_2}\left( \nu^\XX(s_-)\right))}  \PN^{\BG;p}(ds,di_1,di_2,dq,d\theta_1,d\theta_2)\,.
\end{split}
\end{equation}

Taking the expectation in equation \eqref{EQN:Est2} and using estimate \eqref{EQN:Est1} together with
\[
(1+y)^p - y^p \leq C(1+y^{p-1})\,,\quad \forall \, y \geq 0\,.
\]
we have that, for some $C>0$ that can take possibly different values,
\begin{equation}\label{EQN:Est3}
\begin{split}
\EE \sup_{s \in [0,\bar{\tau}_n^\YY]} \norm{\mathbf{1},\nu^\YY(s)}^p &\leq \EE \norm{\mathbf{1},\nu^\YY_0}^p +\\
&+C \, \bar{\AG}\, \bar{\Gamma}^\AG \, \EE \norm{\mathbf{1},\nu^{\XX}_0}\EE \int_0^{\bar{\tau}_n^\YY} \left (1+\norm{\mathbf{1},\nu^{\YY}(s_-)}^{p-1} \right) \norm{\mathbf{1},\nu^{\YY}(s_-)}ds +\\[3pt]
&+ C \, \bar{\BG}\, \bar{\Gamma}^\BG \, \EE \norm{\mathbf{1},\nu^{\XX}_0}\EE \int_0^{\bar{\tau}_n^\YY} \left (1+\norm{\mathbf{1},\nu^{\YY}(s_-)}^{p-1} \right) \norm{\mathbf{1},\nu^{\YY}(s_-)}ds \leq \\
&\leq C \left (1+ \EE \int_0^t \norm{\mathbf{1},\nu^{\YY}(s \wedge \tau_n^\YY)}^p \right )\,.
\end{split}
\end{equation}

From Gronwall lemma it thus follows that it exists $C>0$ depending on p and $T$ but independent of $n$, such that
\begin{equation}\label{EQN:Est4}
\EE \sup_{s \in [0,\bar{\tau}_n^\YY]}\norm{\mathbf{1},\nu^\YY(s)}^p \leq C\,.
\end{equation}

Letting thus $n \to \infty$, we have that 
\begin{equation}\label{EQN:TauInf}
\tau_n^\YY \to \infty\quad \mbox{a.s.}\,.
\end{equation}
In fact, if that was not the case, we can find $T_0<\infty$ such that
\[
\mathbb{P}\left ( \sup_n \tau_n^\YY < T_0\right) = \varepsilon(T_0) >0\,.
\]

This would in turn yields
\[
\EE \sup_{s \in [0,T_0 \wedge \tau_n^\XX]} \norm{\mathbf{1},\nu^\YY(s)}^p \geq \varepsilon(T_0) n^p\,,
\]
which contradicts equation \eqref{EQN:Est4}. A similar argument holds for $\XX$. Using thus Fatou's lemma we can let $n \to \infty$ as
\[
\EE \lim \inf_{n \to \infty} \sup_{s \in [0,T \wedge \tau_n^\XX]}  \norm{\mathbf{1},\nu^\YY(s)}^p \leq \lim \inf_{n \to \infty} \EE \sup_{s \in [0,T \wedge \tau_n^\XX]}  \norm{\mathbf{1},\nu^\YY(s)}^p \leq C < \infty\,,
\]
proving thus \eqref{EQN:SupEst}.

At last, since the above claim holds also for $p=1$, we have that
\begin{equation}\label{EQN:Estp1}
\mathbb{E} \sup_{t \leq T} \norm{\mathbf{1},\nu (t)}  < \infty\,,
\end{equation}
so that the process $\nu$ can be constructed step by step between consecutive jumps and the sequence of jump times $\left (T_m\right)_{m \in \NN}$ goes to infinity and the process in well-defined. The proof is thus complete.
\end{proof}

\begin{Remark}
Notice that if, instead of conditions $(2.ii)-(2.iii)$ in Hypothesis \ref{HYP:1} we require the weaker conditions
\[
\begin{split}
\sup_{q \in \QQ} \sup_{v \in [0,N]} \AG(q,v) < \bar{\AG} < \infty\,,\\
\sup_{q_1,q_2 \in \QQ} \sup_{v \in [0,N]} \BG(q_1,q_2,v) < \bar{\BG} < \infty\,,\\
\end{split}
\]
Theorem \ref{THM:E!} would follow analogously with the only difference that existence and uniqueness can be proved only up to a sufficiently small finite horizon time $T_0<\infty$ rather than on the whole real line $\RR_+$. In particular, equation \eqref{EQN:TauInf} does not hold. To see that, consider the truncated \textit{death rate} and \textit{pairwise interaction rate}
\[
\AG_n(q,v):= \AG(q,v) \Ind{v \leq n}\,,\quad \BG_n(q,v):= \BG(q,v) \Ind{v \leq n}\,.
\]

By the boundedness of the rates $\AG_n$ and $\BG_n$, we have that Theorem \ref{THM:E!} is valid and existence and uniqueness hold true up to a stopping time $\tau_n$. We further clearly have that $\tau_n \leq \tau_{n+1}$, so that, if $\tau_n \to \infty$ as $n \to \infty$, we have the existence and uniqueness for any time horizon $T$, whereas if on the contrary we have that $\tau_n \to T_0$, we have an explosion of the solution in finite time.
\demo
\end{Remark}
      
The next result state a martingale property for the spatial \GSM introduced in previous sections. 
      
\begin{Theorem}\label{THM:MartR}
Assume that Hypothesis \ref{HYP:1} holds true and that $\nu^\XX_0$ and $\nu^\YY_0$ are two random measures independent with finite $p-$th moment, $p \geq 2$. Then:
        
\begin{description}  
\item[(i)] $\nu$ is a Markov process with infinitesimal generator $\LL$ defined by \eqref{EQN:InfGenAll};
     
\item[(ii)] assume that for $F \in C^2_b (\RR \times \RR)$ and for $f^\XX\,,\,f^\YY \in \CC^{2}(\QQ)$ such that for all $\nu \in \MM$, it holds 
\begin{equation}\label{EQN:Mart}
|F_{(f^\XX,f^\YY)}(\nu)| + |\LL F_{(f^\XX,f^\YY)}(\nu)| \leq C \left (1+\norm{\mathbf{1},\nu_0}^p \right)\,.
\end{equation}
      
Then, the process
\begin{equation}\label{EQN:Mart2}
F_{(f^\XX,f^\YY)}(\nu(t)) - F_{(f^\XX,f^\YY)}(\nu_0) - \int_0^t \LL F_{(f^\XX,f^\YY)}(\nu(s))ds\,,
\end{equation}
is a c\'adl\'ag martingale starting at 0;

\item[(iii)] the processes $\Mar^\XX$ and $\Mar^\YY$ defined for $f^\XX\,,\,f^\YY \in C^2_0$ by
\begin{equation}\label{EQN:Mart3}
\begin{cases}
\Mar^\XX(t) &= \norm{f^\XX,\nu^\XX(t)} - \norm{f^\XX,\nu^\XX_0} - \int_0^t \norm{\LL^\XX_d f^\XX (x),\nu^\XX(s)}ds +\\[3pt]
&+ \int_0^t\int_{\QQ} \left[ \RG(q,\nu) + \AG(q,\nu) \right] f^\XX(q)\nu^\XX(s)(dq)ds+\\[3pt]
&+\int_0^t\int_{\tilde{\QQ}^2} \BG(q_1,q_2,\nu) (f^\XX(q_1)+f^\XX(q_2))\nu^\XX(s)(dq_1)\nu^\XX(s)(dq_2)ds+\\[5pt]
\Mar^\YY(t) &= \norm{f^\YY,\nu^\YY(t)} - \norm{f^\YY,\nu^\YY} - \int_0^t \norm{\LL^\YY_d f^\YY (x),\nu^\YY(s)}ds +\\[3pt]
&- \int_0^t \int_{\QQ} \AG(q,\nu) \int_{\QQ} m^{\AG}(\bar{q}|q) f^\YY(\bar{q})d\bar{q} \nu^\XX(s)(dq)ds+\\[3pt]
&- \int_0^t\int_{\tilde{\QQ}^2} \int_{\QQ} \PG(q_1,q_2) \BG(q_1,q_2,\nu) f^\YY(\bar{q}) m^{\BG}(\bar{q}|q_1,q_2)d\bar{q}\nu^\XX(s)(dq_1)\nu^\XX(s)(dq_2)ds\,,
\end{cases}
\end{equation}
are c\'adl\'ag $L^2-$martingale starting at 0 with predictable quadratic variation given by
\begin{equation}\label{EQN:MartQuad}
\begin{cases}
\norm{\Mar^\XX}(t) &= \int_0^t \norm{\nabla^T f^\XX \sigma^\XX \nabla f^\XX , \nu^\XX}ds +\\[3pt]
&+ \int_0^t \int_{\QQ} \left[ \RG(q,\nu(s)) + \AG(q,\nu) \right] \left( f^\XX(q) \right)^2 \nu^\XX(s)(dq)ds+\\[3pt]
& +\int_0^t\int_{\tilde{\QQ}^2} \BG(q_1,q_2,\nu) (f^\XX(q_1)+f^\XX(q_2))^2 \nu^\XX(s)(dq_1)\nu^\XX(s)(dq_2)ds+\\[5pt]
\norm{\Mar^\YY}(t)  &= \int_0^t \norm{\nabla^T f^\YY \sigma^\YY \nabla f^\YY , \nu^\YY}ds +\\[3pt]
&- \int_0^t \int_{\QQ} \AG(q,\nu) \int_{\QQ} m^{\AG}(\bar{q}|q) \left( f^\YY(\bar{q})\right)^2 d\bar{q} \nu^\XX(s)(dq)ds+\\[3pt]
&- \int_0^t\int_{\tilde{\QQ}^2} \int_{\QQ} \PG(q_1,q_2) \BG(q_1,q_2,\nu) \left( f^\YY(\bar{q})\right)^2 m^{\BG}(\bar{q}|q_1,q_2)d\bar{q}\nu^\XX(s)(dq_1)\nu^\XX(s)(dq_2)ds\,.
\end{cases}
\end{equation}
\end{description}
\end{Theorem}

\begin{proof}
\begin{description}
\item[(i)] to show that $\nu$ is a Markov process is standard using the fact Poisson point processes have independent increments. Then, for any function $f^\XX$ and $f^\YY \in C^2_0(\QQ)$, we have the representation given in equation \eqref{EQN:SpatGSM2}. By compensation, we can reformulate equation \eqref{EQN:SpatGSM2} as
\begin{equation}\label{EQN:SpatGSM2Comp2}
\begin{cases}
\norm{f^\XX,\nu^\XX(t)} &= \norm{f^\XX,\nu^\XX(0)} + \int_0^t \norm{\LL^\XX_d f^\XX,\nu^\XX(s)} ds +\\[2pt]
&+ \int_0^t \norm{\left[ \RG(\cdot,\nu) + \AG(\cdot,\nu) \right] f^\XX,\nu^\XX(s)}ds+\\[3pt]
&+\int_0^t \norm{ \norm{ \BG(\cdot,\cdot,\nu(s)) (f^\XX +f^\XX ),\nu^\XX(s)},\nu^\XX(s)}ds +\\[5pt]
\norm{f^\YY,\nu^\YY(t)} &= \norm{f^\YY,\nu^\YY(0)} + \int_0^t \norm{\LL^\YY_d f^\YY,\nu^\YY} ds +\\[2pt]
&- \int_0^t \norm{\AG(\cdot,\nu) \int_{\QQ} m^{\AG}(\bar{q}|\cdot) f^\YY(\bar{q})d\bar{q}, \nu^\XX(s)}ds+\\[3pt]
&- \int_0^t \norm{\norm{\int_{\QQ} \PG(\cdot,\cdot)\BG(\cdot,\cdot) f^\YY(\bar{q}) m^{\BG}(\bar{q}|\cdot,\cdot)d\bar{q},\nu^\XX(s)},\nu^\XX(s)}ds + \tilde{\MM}^\YY(t) \,,
\end{cases}
\end{equation}
where $\tilde{\MM}^\XX$ and $\tilde{\MM}^\YY$ are local-martingales accounting for the noises $W$, $\PN^\RG$, $\PN^\AG$ and $\PN^\BG$. A straightforward computation shows that for $F \in C^2_b (\RR \times \RR)$, dividing equation \eqref{EQN:SpatGSM2Comp2} by $t$, taking the limit as $t \downarrow 0$ and taking the expectation we finally have that $\LL F_{(f^\XX,f^\YY)}(\nu)$ has the expression as given in equation \eqref{EQN:InfGenAll}.

\item[(ii)] using condition \eqref{EQN:Mart} we have can infer that \eqref{EQN:Mart2} is integrable and well-defined. Using point (i) we can finally conclude that \eqref{EQN:Mart2} is a c\'adl\'ag martingale.

\item[(iii)] notice first that point (ii) holds true for any $F_f(\nu) = \norm{f,\nu}^q$, $q \in \{1,\dots,p-1\}$, so that, choosing $q=1$ we immediately have that $\Mar^\XX(t)$ and $\Mar^\YY(t)$ are martingales. Using thus $p=2$ we obtain computing $F_f(\nu) = \norm{f,\nu}^2$ from equations \eqref{EQN:InfDiff}-\eqref{EQN:InfReact},
{\footnotesize
\begin{equation}\label{EQN:QuadV1}
\begin{cases}
&\norm{f^\XX,\nu^\XX(t)}^2 - \norm{f^\XX,\nu^\XX(0)}^2 +\\[3pt] 
& - \int_0^t \int_\QQ \left[ 2\norm{f^\XX,\nu^\XX(s)} \left(\mu^\XX \left(q,\norm{\Gamma^{\mu,\XX}_{q},\nu}\right) \cdot \nabla f^\XX\left(q\right)ds + \frac{1}{2} Tr\left [\Sigma^\XX\left(q,\norm{\Gamma^{\sigma,\XX}_{q},\nu}\right) \mbox{Hess} \,f^\XX\left(q\right)\right ]\right) \right]\nu^\XX(s)(dq)ds+\\[3pt]
&-\int_0^t \int_\QQ 2 \nabla^T f^\XX(q)\, \sigma^\XX \,\nabla f^\XX(q) \nu^\XX(s)(dq)ds+\\[3pt]
&+\int_0^t \int_\QQ \RG(q,\nu)\left[\left (f^\XX(q)\right)^2 - 2 f^\XX(q) \norm{f^\XX,\nu^\XX(s)} \right]\nu^\XX(s)(dq)ds+\\[3pt]
&+ \int_0^t \int_{\QQ} \AG(q,\nu)\left[\left (f^\XX(q)\right)^2 - 2 f^\XX(q) \norm{f^\XX,\nu^\XX(s)} \right]\nu^\XX(s)(dq)ds+\\[3pt]
&+ \int_0^t \int_{\tilde{\QQ}^2}\left[\left (f^\XX(q_1)\right)^2 + \left (f^\XX(q_1)\right)^2 +f^\XX(q_1)f^\XX(q_2) - 2 f^\XX(q_1) \norm{f^\XX,\nu^\XX(s)}- 2 f^\XX(q_2) \norm{f^\XX,\nu^\XX(s)} \right] \times \\[3pt]
&\qquad \qquad\qquad \qquad \times \BG(q_1,q_2,\nu) \nu^\XX(s)(dq_1)\nu^\XX(s)(dq_2) \,ds+\\[5pt]
&\norm{f^\YY,\nu^\YY(t)}^2 - \norm{f^\YY,\nu^\YY(0)}^2 +\\[3pt] 
& - \int_0^t \int_\QQ \left[ 2\norm{f^\YY,\nu^\YY(s)} \left(\mu^\YY \left(q,\norm{\Gamma^{\mu,\YY}_{q},\nu}\right) \cdot \nabla f^\YY\left(q\right)ds + \frac{1}{2} Tr\left [\Sigma^\YY\left(q,\norm{\Gamma^{\sigma,\YY}_{q},\nu}\right) \mbox{Hess} \,f^\YY\left(q\right)\right ]\right) \right]\nu^\YY(s)(dq)ds+\\[2pt]
&-\int_0^t \int_\QQ 2 \nabla^T f^\YY(q)\, \sigma^\YY \,\nabla f^\YY(q) \nu^\YY(s)(dq)ds+\\[3pt]
&-\int_0^t \int_{\QQ} \int_{\QQ} \AG(q,\nu)\left[\left (f^\YY(\bar{q})\right)^2 + 2 f^\YY(\bar{q}) \norm{f^\YY,\nu^\YY(s)} \right]m^{\AG}(\bar{q}|q)d\bar{q}\nu^\XX(s)(dq)ds+\\[2pt]
&-\int_0^t \int_{\tilde{\QQ}^2} \int_{\QQ}  \PG(q_1,q_2)\BG(q_1,q_2,\nu)\left[\left (f^\YY(\bar{q})\right)^2 + 2 f^\YY(\bar{q}) \norm{f^\YY,\nu^\YY(s)} \right]m^{\BG}(\bar{q}|q_1,q_2)d\bar{q}\nu^\XX(s)(dq_1)\nu^\XX(s)(dq_2)ds\,.
\end{cases}
\end{equation}
}

At the same time, computing $\norm{f,\nu}^2$ via It\^{o} formula from equation \eqref{EQN:Mart}, yields that
{\footnotesize
\begin{equation}\label{EQN:QuadV2}
\begin{cases}
&\norm{f^\XX,\nu^\XX(t)}^2 - \norm{f^\XX,\nu^\XX_0}^2 +\\[3pt]
&- \int_0^t \int_\QQ \left[ 2\norm{f^\XX,\nu^\XX(s)} \left(\mu^\XX \left(q,\norm{\Gamma^{\mu,\XX}_{q},\nu}\right) \cdot \nabla f^\XX\left(q\right)ds + \frac{1}{2} Tr\left [\Sigma^\XX\left(q,\norm{\Gamma^{\sigma,\XX}_{q},\nu}\right) \mbox{Hess} \,f^\XX\left(q\right)\right ]\right) \right]\nu^\XX(s)(dq) +\\[3pt]
&+ \int_0^t\int_{\QQ} \left[ \RG(q,\nu) + \AG(q,\nu)\right] f^\XX(q)\norm{f^\XX,\nu^\XX(s)}\nu^\XX(s)(dq)ds+\\[3pt]
&+\int_0^t\int_{\tilde{\QQ}^2}\BG(q_1,q_2,\nu) (f^\XX(q_1)+f^\XX(q_2))\norm{f^\XX,\nu^\XX(s)}\nu^\XX(s)(dq_1)\nu^\XX(s)(dq_2)ds+\\[5pt]
&\norm{f^\YY,\nu^\YY(t)} - \norm{f^\YY,\nu^\YY_0} +\\[3pt]
&- \int_0^t \int_\QQ \left[ 2\norm{f^\YY,\nu^\YY(s)} \left(\mu^\YY \left(q,\norm{\Gamma^{\mu,\YY}_{q},\nu}\right) \cdot \nabla f^\YY\left(q\right)ds + \frac{1}{2} Tr\left [\Sigma^\YY\left(q,\norm{\Gamma^{\sigma,\YY}_{q},\nu}\right) \mbox{Hess} \,f^\YY\left(q\right)\right ]\right) \right]\nu^\YY(s)(dq) ds+\\[3pt]
&- \int_0^t \int_{\QQ} \AG(q,\nu) \int_{\QQ} m^{\AG}(\bar{q}|q) f^\YY(\bar{q})d\bar{q} \norm{f^\YY,\nu^\YY(s)}\nu^\XX(s)(dq)ds+\\[3pt]
&- \int_0^t\int_{\tilde{\QQ}^2} \int_{\QQ} \PG(q_1,q_2) \BG(q_1,q_2,\nu) f^\YY(\bar{q}) m^{\BG}(\bar{q}|q_1,q_2)d\bar{q}\norm{f^\YY,\nu^\YY(s)} \nu^\XX(s)(dq_1)\nu^\XX(s)(dq_2)ds - \Mar^\YY(t)\,,
\end{cases}
\end{equation}
}
is a c\'adl\'ag martingale. Comparing equations \eqref{EQN:QuadV1} and \eqref{EQN:QuadV2} implies equation \eqref{EQN:MartQuad}.
\end{description}
\end{proof}

\subsection{On the initial distribution}\label{SEC:InitSpat}

As clear by the description given in Section \ref{SEC:Init}, the initial distribution considered lacks any spatial distribution on the dose and on the formation of the lesion in the cell nucleus. To compute the initial lesion distribution $\nu_0$, we need to generalize the treatment given in Section \ref{SEC:Init} to include in equation \eqref{EQN:InitDistGSM2} a spatial description. A possible mathematical formulation of the initial lesion computation would be the following. For a better understanding, we will provide a step-wise construction of such distribution: 
\begin{description}
\item[(i)] given a certain dose $D$ and fluence average specific energy $z_F$, a random number of events $\nu$ in a cell nucleus is sampled from a distribution $p_e$. A typical assumption would be, due to the independence of events, to assume $p_e$ a Poisson distribution with average $\frac{D}{z_F}$;\\

\item[(ii)] the $\nu$ events are distributed randomly over the cell nucleus. Under an isotropic and uniform random field, the distribution can be assumed to be uniform over the domain, or in a more general setting, the distribution can be sampled from an \textit{a priori} calculated distribution of tracks using a \textit{condensed history} MC code, \cite{agostinelli2003geant4}. A similar distribution has been for instance calculated in \cite{missiaggia2021novel,missiaggia2022exploratory};\\

\item[(iii)] for any event $i$, $i=1,\dots,\nu$, a certain specific energy $z_i$ is sampled according to the single-event microdosimetric specific energy distribution $f_1(z)$;\\

\item[(iv)] for any event $i$, $i=1,\dots,\nu$, with specific energy deposition $z_i$, the number $\xi^\XX_i$ and $\xi^\YY_i$ of sub-lethal and lethal lesion respectively is sampled from a distribution $p$. A typical assumption would be to assume such distribution the product of two independent Poisson distributions of average $\kappa(z_i)$ and $\lambda(z_i)$ respectively, for some suitable functions $\kappa$ and $\lambda$;\\

\item[(v)] denote by $\xi^\XX := \sum_{i=1}^\nu \xi^\XX_i$ and $\xi^\YY := \sum_{i=1}^\nu \xi^\YY_i$ the number of sub-lethal and lethal lesion respectively. Thus sample the positions $(q^\XX,q^\XX) \in \QQ^{|(\xi^\XX,\xi^\YY)|}$, according to a distribution $\zeta_\xi(\left . q^\XX,q^\YY \right| \xi^\XX,\xi^\YY)$. A reasonable choice for such distribution would be to distribute $z_i$ spatially around the track using the Amorphous Track model, \cite{kase2007biophysical}, which is a parametrization of the radial dose distribution around a particle track. In particular, denoting by $AT_i(q)$ the normalized radial dose distribution representing the probability of depositing a certain dose in a domain, for any $\QQ_1 \subset \QQ$, the relative dose absorbed in $\QQ_1$ is thus given by 
\[
z_i\int_{\QQ_1}AT_i(q)dq\,.
\]
Then, the probability density distribution describing the probability of creating a lesion in $\QQ_1$ is given by
\[
\sum_{i=1}^\nu z_i \int_{\QQ_1}AT_i(q)dq\,.
\]
\end{description}

\begin{Remark}
As mentioned in the introduction, a further choice would be to use track structure code to simulate the spatial distribution of lesions within a cell nucleus. Several papers have been published in the literature showing how this can be achieved, \cite{chatzipapas2022simulation,kyriakou2022review,zhu2020cellular,thibaut2023minas}. Nonetheless, none of these papers then asses the biological effect of given radiation using a true spatial biological model, so the accuracy in the description of the geometry of the biological target is lost. 
\end{Remark}

\section{The protracted irradiation}\label{SEC:ProtI}

In the present Section, we assume a further rate besides the interaction rates described in Section \ref{SEC:SpatGSM2}. Such a rate accounts for the formation of a new random number of both lethal and sub-lethal lesions due to protracted irradiation. We thus have the following system of possible pathways
\begin{equation}\label{EQN:Rates2}
\begin{cases}
& \XX \xrightarrow{\RG} \emptyset\,,\\
& \XX \xrightarrow{\AG} Y\,,\\
& \XX + \XX \xrightarrow{\BG} 
\begin{cases}
\YY & \mbox{with probability} \quad p \in [0,1]\,,\\
\emptyset & \mbox{with probability} \quad 1-p \in [0,1]\,,\\
\end{cases}\,,\\
& \emptyset \xrightarrow{\DG} 
\begin{cases}
\xi^\XX \,\XX\,,\\
\xi^\YY \,\YY\,,
\end{cases}\,.
\end{cases}
\end{equation}
where $\xi^\XX$ and $\xi^\YY$ are two suitable (possibly correlated) $\mathbb{N}-$valued random variables. The last pathway, namely the dose-rate $\DG$ is reasonable to be assumed strictly positive up to a certain time horizon $T_{irr}$, representing the irradiation period.

We thus have the following:
\begin{description}
\item[(v) - protracted irradiation] at a certain \textit{dose rate} $\DG$ a random number $\xi^\XX$ and $\xi^\YY$ sub-lethal and lethal lesions, respectively, are created in $\QQ$. This process can be described by a spatial compound random measure 
\[
\zeta = (\zeta^\XX,\zeta^\YY) = \left(\sum_{i=0}^{\xi^\XX} \delta_{\QQ_i},\sum_{i=0}^{\xi^\YY} \delta_{\QQ_i} \right) \in \MM \times \MM\,.
\]
We assume that the random measure $\zeta$ admits a probability measure of the form
\[
\zeta_\xi(\left . q^\XX,q^\YY \right| \xi^\XX,\xi^\YY) p(\xi^\XX,\xi^\YY)\,,
\]
with $p(\xi^\XX,\xi^\YY)$ a discrete probability distribution on $\NN^2$ representing the probability of inducing $(\xi^\XX,\xi^\YY)$ sub-lethal and lethal lesions and $\zeta_\xi$ a spatial distribution representing the probability of creating $(\xi^\XX,\xi^\YY)$ sub-lethal and lethal lesions at positions $( q^\XX,q^\YY) \in \QQ^{|(\xi^\XX,\xi^\YY)|}$. We will further denote for short the marginal distributions by
\[
\begin{split}
\zeta_{\xi^\XX}(\left . q^\XX \right| \xi^\XX) \,,\quad p_{\XX}(\xi^\XX)\,,\\
\zeta_{\xi^\YY}(\left . q^\YY \right| \xi^\YY) \,,\quad p_{\YY}(\xi^\YY)\,.
\end{split}
\]

The \textit{protracted irradiation rate} $\DG$ is associated to a Poisson point measure 
\[
\PN^{\DG}(ds,d\xi^\XX,d\xi^\YY,dq,d\theta_1,d\theta_2) \quad \mbox{on} \quad \RR_+ \times \NN^2 \times \QQ^{|(\xi^\XX,\xi^\YY)|} \times \RR \times \RR\,.
\]

The corresponding intensity measure associated to $\PN^{\DG}$ is
\[
 \lambda^{\DG} (ds,d\xi^\XX,d\xi^\YY,dq,d\theta_1,d\theta_2) := ds \otimes dp(\xi^\XX,\xi^\YY) \otimes dq \otimes d\theta_1 \otimes d\theta_2\,.
\]

We denote with $\PNC^\BG$ the compensated Poisson measure defined as
\[
\PNC^{\DG}(ds,d\xi^\XX,d\xi^\YY,dq,d\theta_1,d\theta_2) := \PN^{\DG}(ds,d\xi^\XX,d\xi^\YY,dq,d\theta_1,d\theta_2) - \lambda^{\DG} (ds,d\xi^\XX,d\xi^\YY,dq,d\theta_1,d\theta_2)\,.
\]
\end{description}

\begin{Remark}
The protracted irradiation can be interpreted as an improved description of the initial distribution. In particular, if $\DG$ is sufficiently large and $T_{\mbox{irr}}$ is sufficiently small, only the creation of new damages due to the protracted irradiation happens before any of the other pathways can happen. This is typically the case in the clinical scenario, where the \textit{dose rate} usually dominates the biological interaction rates; such a situation is referred to as \textit{conventional dose rate}. In this case, it is reasonable to assume that the initial distribution of lesions $\nu^\XX_0$ and $\nu^\YY_0$ in the instantaneous irradiation, that is $\DG = 0$, coincides with the distribution of lesions under protracted irradiation at time $T_0$. For this reason, a typical distribution $\zeta$ can be obtained following the description provided in Section \ref{SEC:Init} in the particular case of a single event hitting the domain, that is $\nu =1$. It is further worth stressing that there are some relevant situations where an explicit treatment of the effect of protracted irradiation can play a relevant role: (i) a \textit{split dose} irradiation treatment, where the treatment is split into several treatments with a smaller dose to favorite normal tissue recovery between treatments, (ii) space radioprotection, characterized by extremely low dose rates exposure over a long period and (iii) FLASH radiotherapy. Both (i) and (ii) are situations where is fundamental to model the entire spatial distribution of radiation-induced damages over a relatively long time period so that the inclusion of a specific protracted irradiation rate is necessary. For what instead concern (iii), it will be explicitly treated in Section \ref{SEC:FLASH}.
\demo
\end{Remark}

We will assume the following to hold.

\begin{Hypothesis}\label{HYP:2}
\begin{enumerate}
\setItemnumber{4}
\item protracted dose rate components:
\begin{description}
\item[(4.i)] the \textit{protracted irradiation rate} $\DG$ is positive and finite;\\
\item[(4.ii)] for any $\xi^\XX\,,\xi^\YY \in \NN$, $\zeta_\xi$ is a probability measure, i.e.
\[
\int_{\QQ^{|(\xi^\XX,\xi^\YY)|}} \zeta_\xi(\left . q^\XX,q^\YY \right| \xi^\XX,\xi^\YY)dq^\XX \, dq^\YY =1\,;
\]
\item[(4.iii)] the random measure $p$ admits finite $p-$moments, that is, for $p \geq 1$, it holds
\[
\int_{\NN^2} \left (\xi^\XX \right )^p\,dp(\xi^\XX,\xi^\YY) < \infty \,, \quad \int_{\NN^2} \left (\xi^\YY \right )^p\, dp(\xi^\XX,\xi^\YY)  < \infty \,.
\]
\end{description}
\end{enumerate}
\end{Hypothesis}

Therefore, the resulting process is characterized by the process defined in equation \eqref{EQN:SpatGSM2} with the addition of the random measure $\PN^{\DG}$. In particular, denote for short by $\LL^\XX_B$ and $\LL^\YY_B$ the process introduced in equation \eqref{EQN:SpatGSM2}, then the process under the effect of protracted irradiation is characterized by the following weak representation, given $f^\XX$ and $f^\YY \in C^2_0(\RR)$, we have
\begin{equation}\label{EQN:SpatGSM2PI}
\begin{cases}
\norm{f^\XX,\nu^\XX(t)} &= \LL^\XX_B \nu^\XX(t) +\\
&+ \int_0^t \int_{\NN^2} \int_{\QQ^{|(\xi^\XX,\xi^\YY)|}} \int_{\RR}\int_{\RR} \left[ \norm{f^\XX,\nu^{\XX}(s_-) + \sum_{i=1}^{\xi^\XX} \delta_{q_i}} - \norm{f^\XX,\nu^{\XX}(s_-)}\right] \times\\[3pt]
&\qquad \times \Ind{\theta_1 \leq \DG}\Ind{\theta_2 \leq \zeta_\xi^\XX(q|\xi^\XX)} \PN^{\DG}(ds,d\xi^\XX,d\xi^\YY,dq,d\theta_1,d\theta_2)\,,\\[5pt]
\norm{f^\YY,\nu^\YY(t)} &= \LL^\YY_B \nu^\YY(t) +\\
&+\int_0^t \int_{\NN^2} \int_{\QQ^{|(\xi^\XX,\xi^\YY)|}} \int_{\RR}  \int_{\RR} \left[ \norm{f^\YY,\nu^{\YY}(s_-) + \sum_{i=1}^{\xi^\YY} \delta_{q_i}} - \norm{f^\YY,\nu^{\YY}(s_-)}\right] \times\\[3pt]
&\qquad \times \Ind{\theta_1 \leq \DG} \Ind{\theta_2 \leq \zeta_\xi^\YY(q|\xi^\YY)} \PN^{\DG}(ds,d\xi^\XX,d\xi^\YY,dq,d\theta_1,d\theta_2)\,.
\end{cases}
\end{equation}

We thus augment the probability space with the processes defined in $(v)$. We thus have the following definition.

\begin{Definition}\label{DEF:DefNu2}
We say that $\nu(t) = (\nu^\XX(t),\nu^\YY(t))$ as defined in equations \ref{DEF:Nu1}--\ref{DEF:Nu2} is a \textit{spatial radiation-induced DNA damage repair model under protracted irradiation} if $\nu = \left(\nu(t)\right)_{t \in \RR_+}$ is $\left (\mathcal{F}_t\right )_{t \in \RR_+}-$adapted and for any $f^\XX$ and $f^\YY\in C^2_0(\QQ)$ equation \eqref{EQN:SpatGSM2PI} holds $\mathbb{P}-$a.s.
\end{Definition}

We thus have the following well-posedness result.

\begin{Theorem}\label{THM:E!2}
Let $\nu^\XX_0$ and $\nu^\YY_0$ two random measures with finite $p-$th moment, $p \geq 1$, that is it holds
\begin{equation}\label{EQN:InitNu2}
\EE \norm{\mathbf{1},\nu^\XX_0}^p < \infty\,,\quad \EE \norm{\mathbf{1},\nu^\YY_0}^p < \infty\,.
\end{equation}

Then, under Hypothesis \ref{HYP:1}-\ref{HYP:2}, for any $T>0$ it exists a unique strong solution in $\DD \left ([0,T],\MM \times \MM \right)$ to the system \eqref{EQN:SpatGSM2}. Also, it holds
\begin{equation}\label{EQN:SupEst2}
\mathbb{E} \sup_{t \leq T} \norm{\mathbf{1},\nu (t)}^p  < \infty\,.
\end{equation}
\end{Theorem}
\begin{proof}
The proof proceeds with similar arguments as in the proof of Theorem \ref{THM:E!}, taking also into account the protracted irradiation term.

For $t \geq 0$, we have,
\begin{equation}\label{EQN:Est1PI}
\begin{split}
\sup_{s \in [0,\bar{\tau}_n^\XX]} \norm{\mathbf{1},\nu^\XX(s)}^p & \leq \norm{\mathbf{1},\nu^\XX_0}^p + \\
&+\int_0^{\bar{\tau}_n^\XX}\int_{\NN^2} \int_{\QQ^{|(\xi^\XX,\xi^\YY)|}} \int_\RR \int_\RR \left[ \norm{f^\XX,\nu^{\XX}(s_-) + \sum_{i=1}^{\xi^\XX} \delta_{q_i}} - \norm{f^\XX,\nu^{\XX}(s_-)}\right] \times\\
&\qquad \times \Ind{\theta_1 \leq \DG}\Ind{\theta_2 \leq \zeta_\xi^\XX(q|\xi^\XX)} \PN^{\DG}(ds,d\xi^\XX,d\xi^\YY,dq,d\theta_1,d\theta_2)\,.
\end{split}
\end{equation}

Notice that, for any positive integer $x$ and $y$ it holds
\begin{equation}\label{EQN:EstNewPI}
(x+y)^p - y^p \leq C_{p}y^{p-1}x^{p}\,,
\end{equation}
so that, using equation \eqref{EQN:EstNewPI} into equation \eqref{EQN:Est1PI}, we have that, for some $C>0$ that can take possibly different values,
\begin{equation}\label{EQN:Est3PI}
\begin{split}
\EE \sup_{s \in [0,\bar{\tau}_n^\XX]} \norm{\mathbf{1},\nu^\XX(s)}^p &\leq \EE \norm{\mathbf{1},\nu^\XX_0}^p +\\
&+C \DG \EE \int_0^{\bar{\tau}_n^\XX} \int_{\NN^2}\int_{\QQ^{|(\xi^\XX,\xi^\YY)|}} \left(\xi^\XX\right)^p \norm{\mathbf{1},\nu^{\XX}(s_-)}^{p-1}\zeta_\xi(\left . q^\XX,q^\YY \right| \xi^\XX,\xi^\YY) dq^\XX\,dq^\YY \,dp(\xi^\XX,\xi^\YY)ds \leq \\
&\leq C \left (1+ \EE \int_0^t \norm{\mathbf{1},\nu^{\XX}(s \wedge \tau_n^\YY)}^p \right )\,.
\end{split}
\end{equation}
Gronwall lemma implies that there exists $C>0$ depending on p and $T$ but independent of $n$, such that
\begin{equation}\label{EQN:Est4PI}
\EE \sup_{s \in [0,\bar{\tau}_n^\XX]}\norm{\mathbf{1},\nu^\XX(s)}^p \leq C\,.
\end{equation}
Similar arguments can be used to prove that 
\begin{equation}\label{EQN:Est5PI}
\EE \sup_{s \in [0,\bar{\tau}_n^\YY]}\norm{\mathbf{1},\nu^\YY(s)}^p \leq C\,.
\end{equation}
The proof thus follows in a straightforward manner proceeding as in the proof of Theorem \ref{THM:E!}.
\end{proof}

The above process is characterized by the following infinitesimal generator
\begin{equation}\label{EQN:InfGenAlldProt}
\LL F_{(f^\XX,f^\YY)}(\nu) = \LL_d F_{(f^\XX,f^\YY)}(\nu) + \sum_{h \in \{\RG,\AG,\BG,\DG\}} \LL_{h} F_{(f^\XX,f^\YY)}(\nu)\,,    
\end{equation}
where $\LL_d F_{(f^\XX,f^\YY)}(\nu)$ and $\LL_h F_{(f^\XX,f^\YY)}(\nu)$, $h \in \{\RG,\AG,\BG,\DG\}$, are the infinitesimal generators introduced in equations \eqref{EQN:InfDiff}--\eqref{EQN:InfReact}, whereas $\LL_{\DG}$ is defined as
\begin{equation}\label{EQN:InfDRProt}
\begin{split}
\LL_{\DG} F_{(f^\XX,f^\YY)}(\nu) &= \DG\int_{\NN^2} \int_{\QQ^{|(\xi^\XX,\xi^\YY)|}} \left[F_{(f^\XX,f^\YY)}\left (\nu^\XX+\sum_{i=1}^{\xi^\XX}\delta_{q_i^\XX},\nu^\YY+\sum_{i=1}^{\xi^\YY}\delta_{q_i^\YY}\right ) - F_{(f^\XX,f^\YY)}(\nu) \right] \times \\
&\qquad \qquad \qquad \qquad \times \zeta_\xi(\left . q^\XX,q^\YY \right| \xi^\XX,\xi^\YY)dq^\XX\,dq^\YY dp(\xi^\XX,\xi^\YY) \,.\\
\end{split}
\end{equation}

We thus have the martingale properties and representation corresponding to Theorem \ref{THM:MartR}.

\begin{Theorem}\label{THM:MartR2}
Assume that Hypothesis \ref{HYP:1}-\ref{HYP:2} holds true and that $\nu^\XX_0$ and $\nu^\YY_0$ are two random measures independent with finite $p-$th moment, $p \geq 2$. Then:
        
\begin{description}  
\item[(i)] $\nu$ is a Markov process with infinitesimal generator $\LL$ defined by \eqref{EQN:InfGenAlldProt};
     
\item[(ii)] assume that for $F \in C^2_b (\RR \times \RR)$ and for $f^\XX\,,\,f^\YY \in \CC^{2}(\QQ)$ such that for all $\nu \in \MM$, it holds 
\begin{equation}\label{EQN:Mart1a}
|F_{(f^\XX,f^\YY)}(\nu)| + |\LL F_{(f^\XX,f^\YY)}(\nu)| \leq C \left (1+\norm{\mathbf{1},\nu_0}^p \right)\,.
\end{equation}
      
Then, the process
\begin{equation}\label{EQN:Mart21a}
F_{(f^\XX,f^\YY)}(\nu(t)) - F_{(f^\XX,f^\YY)}(\nu_0) - \int_0^t \LL F_{(f^\XX,f^\YY)}(\nu(s))ds\,,
\end{equation}
is a c\'adl\'ag martingale starting at 0;

\item[(iii)] the processes $\Mar^\XX$ and $\Mar^\YY$ defined for $f^\XX\,,\,f^\YY \in C^2_0$ by
\begin{equation}\label{EQN:Mart31a}
\begin{cases}
\Mar^\XX(t) &= \norm{f^\XX,\nu^\XX(t)} - \norm{f^\XX,\nu^\XX_0} - \int_0^t \norm{\LL^\XX_d f^\XX (x),\nu^\XX(s)}ds +\\[3pt]
&+ \int_0^t\int_{\QQ} \left[ \RG(q,\nu) + \AG(q,\nu) \right] f^\XX(q)\nu^\XX(s)(dq)ds+\\[3pt]
&+\int_0^t\int_{\tilde{\QQ}^2} \BG(q_1,q_2,\nu) (f^\XX(q_1)+f^\XX(q_2))\nu^\XX(s)(dq_1)\nu^\XX(s)(dq_2)ds+\\[3pt]
&-\int_0^t \int_{\NN} \int_{\tilde{\QQ}^{\xi^\XX}} \DG \left (\sum_{i=1}^{\xi^\XX}f^\XX(q_i^\XX) \right) \zeta_{\xi^\XX}(\left . q^\XX \right| \xi^\XX) dq^\XX dp(\xi^\XX,\xi^\YY) ds\,,\\[5pt]
\Mar^\YY(t) &= \norm{f^\YY,\nu^\YY(t)} - \norm{f^\YY,\nu^\YY} - \int_0^t \norm{\LL^\YY_d f^\YY (x),\nu^\YY(s)}ds +\\[3pt]
&- \int_0^t \int_{\QQ} \AG(q,\nu) \int_{\QQ} m^{\AG}(\bar{q}|q) f^\YY(\bar{q})d\bar{q} \nu^\XX(s)(dq)ds+\\[3pt]
&- \int_0^t\int_{\tilde{\QQ}^2} \int_{\QQ} \PG(q_1,q_2) \BG(q_1,q_2,\nu) f^\YY(\bar{q}) m^{\BG}(\bar{q}|q_1,q_2)d\bar{q}\nu^\XX(s)(dq_1)\nu^\XX(s)(dq_2)ds+\\[3pt]
&-\int_0^t \int_{\NN} \int_{\tilde{\QQ}^{\xi^\YY}} \DG \left (\sum_{i=1}^{\xi^\YY} f^\YY(q_i^\YY) \right) \zeta_{\xi^\YY}(\left .q^\YY \right| \xi^\YY) dq^\YY dp(\xi^\XX,\xi^\YY) ds\,,
\end{cases}
\end{equation}
are c\'adl\'ag $L^2-$martingale starting at 0 with predictable quadratic variation given by
\begin{equation}\label{EQN:MartQuad1a}
\begin{cases}
\norm{\Mar^\XX}(t) &= \int_0^t \norm{\nabla^T f^\XX \sigma^\XX \nabla f^\XX , \nu^\XX}ds +\\[3pt]
&+ \int_0^t \int_{\QQ} \left[ \RG(q,\nu) + \AG(q,\nu)\right] \left( f^\XX(q) \right)^2 \nu^\XX(s)(dq)ds+\\[3pt]
& +\int_0^t\int_{\tilde{\QQ}^2} \BG(q_1,q_2,\nu) (f^\XX(q_1)+f^\XX(q_2))^2 \nu^\XX(s)(dq_1)\nu^\XX(s)(dq_2)ds+\\[3pt]
&-\int_0^t \int_{\NN} \int_{\tilde{\QQ}^{\xi^\XX}} \DG \left (\sum_{i=1}^{\xi^\XX}f^\XX(q_i^\XX) \right)^2 \zeta_{\xi^\XX}(\left . q^\XX \right| \xi^\XX)dq^\XX dp(\xi^\XX,\xi^\YY)  ds\,,\\[5pt]
\norm{\Mar^\YY}(t)  &= \int_0^t \norm{\nabla^T f^\YY \sigma^\YY \nabla f^\YY , \nu^\YY}ds +\\[3pt]
&- \int_0^t \int_{\QQ} \AG(q,\nu) \int_{\QQ} m^{\AG}(\bar{q}|q) \left( f^\YY(\bar{q})\right)^2 d\bar{q} \nu^\XX(s)(dq)ds+\\[3pt]
&- \int_0^t\int_{\tilde{\QQ}^2} \int_{\QQ} \PG(q_1,q_2) \BG(q_1,q_2,\nu) \left( f^\YY(\bar{q})\right)^2 m^{\BG}(\bar{q}|q_1,q_2)d\bar{q}\nu^\XX(s)(dq_1)\nu^\XX(s)(dq_2)ds+\\[3pt]
&-\int_0^t \int_{\NN} \int_{\tilde{\QQ}^{\xi^\YY}} \DG \left (\sum_{i=1}^{\xi^\YY} f^\YY(q_i^\YY) \right)^2 \zeta_{\xi^\YY}(\left . q^\YY \right| \xi^\YY)dq^\YY dp(\xi^\XX,\xi^\YY) ds\,,
\end{cases}
\end{equation}
\end{description}
\end{Theorem}
\begin{proof}
The proof is analogous to the proof of Theorem \ref{THM:MartR}.
\end{proof}

\subsection{The bio-chemical system under protracted irradiation}\label{SEC:FLASH}

As mentioned above, before focusing on a macroscopic limit of the spatial DNA repair model, we will consider a different setting relevant to the considered application. As commented in Section \ref{SEC:Init}, the functions $\kappa$ and $\lambda$ usually include information regarding the chemical environment and radical formation. In the following treatment, we make this assumption explicit, assuming that the formation of new damages depends on the chemical environment described by a set of reaction-diffusion equations. It is worth stressing that in general, the energy deposition of the particle also affects the chemical environment so the above-mentioned reaction-diffusion equation also includes a term dependent on the energy deposition. As the chemical evolves on a much faster time scale, the concentration of chemical species will be described by a set of parabolic PDE, with a random discontinuous inhomogeneous term due to the effect of radiation. We will not consider a specific model for the chemical environment, but on the contrary, we will assume a general \textit{mass control hypothesis} that includes many possible systems proposed in the literature. Future investigation will be specifically devoted to analyzing and implementing the highly dimensional chemical system, including the homogeneous chemical stage also the heterogeneous one.

Assume a set of $L$ chemical species, then, for $i=1,\dots,L$, the concentration of the $i-$th species $\rho_i$ evolves according to

\begin{equation}\label{EQN:Chemistry}
\begin{cases}
\frac{\partial}{\partial t} \rho_i(q, t)= D_i \Delta_q \rho_i(q, t) + f_i(\rho) + G_i(q)\,, & \quad \mbox{ in } \quad  \QQ \times [0,T]\,,\\
\nabla_q \rho_i \cdot n(q) = 0 \,, & \quad \mbox{ in } \quad  \partial \QQ \times (0,T)\,,\\
\rho_i(0,q)= \rho_{0;i}(q) \,, & \quad \mbox{ in } \quad  \QQ \,.\\
\end{cases}
\end{equation}

We consider the following.

\begin{description}
\item[(v') - protracted irradiation] at a certain \textit{dose rate} $\DG$ a random number $\xi^\XX$ and $\xi^\YY$ sub-lethal and lethal lesions, respectively, are created in $\QQ$. This process is described by a random measure that depends on chemical concentration
\[
\zeta = (\zeta^\XX,\zeta^\YY) = \left(\sum_{i=0}^{\xi^\XX} \delta_{\QQ_i},\sum_{i=0}^{\xi^\YY} \delta_{\QQ_i} \right) \in \MM \times \MM\,.
\]
We assume that the random measure $\zeta$ admits a decomposition of the form
\[
\zeta_\xi(\left . q^\XX,q^\YY \right| \xi^\XX,\xi^\YY) p\left (\left . \xi^\XX,\xi^\YY\right | \rho(q, t) \right)\,,
\]
with $p(\xi^\XX,\xi^\YY)$ a discrete probability distribution on $\NN^2$.

\item[(vi) - chemical environment] for all $i=1,\dots,L$, the random discontinuous inhomogeneous term $G$ is defined as
\[
G_i : \Omega \times \QQ \to \RR_+\,,\quad \Omega \times \QQ \ni (\omega,q) \mapsto \sum_{k=1}^{\infty} \ZZ^{k;i}(q,\omega)\delta_{\tau^k_{\DG(\omega)}}\,. 
\]
\end{description}

\begin{Remark}
The random function $Z^i$ represents the energy deposited by an event and can be computed as described in Section \ref{SEC:Init}. Regarding instead the dependence of $p$ on the chemical environment, several possible choices can be made. What is currently known is that the actual number of damage created by certain radiation depends on the chemical environment. Therefore, considering the description of the initial damage distribution given in Section \ref{SEC:Init}, a meaningful choice would be to assume that given a certain specific energy deposition, the average number of induced lethal and sub-lethal lesions described in $(iv)$ depends on the chemical environment, that is we have $\lambda(z_i,\rho)$ and $\kappa(z_i,\rho_i)$.
\demo
\end{Remark}

We will assume the following to hold.

\begin{Hypothesis}\label{HYP:3}

\begin{enumerate}
\setItemnumber{5}
\item chemical environment components:
\begin{description}
\item[(5.i)] for all $i=1,\dots,L$, the random initial condition $\rho_{0;i}$ is bounded and non-negative $\mathbb{P}-$a.s., that is
\[
\rho_{0;i}(q,\omega) \in L^\infty(\QQ)\,,\quad \mbox{and} \quad \rho_{0;i}(q,\omega) \geq 0\,, \quad \mbox{ for a.e. } q \in \QQ\,,\quad \mathbb{P}-\mbox{a.s.}\,,
\]
and has finite $p-$th moment, $p \geq 1$, that is
\[
\EE \left \| \rho_{0;i}\right\|^p_\infty < \infty\,.
\]

\item[(5.ii)] there exist constants $C_0$ and $C_1$ such that
\[
\sum_{i=1}^L f_i(\rho) \leq C_0 + C_1 \sum_{i=1}^L \rho_i\,;
\]

\item[(5.iii)] for all $i=1,\dots,L$, $f_i$ is locally Lipschitz and
\[
f_i(\rho) \geq 0 \,,\quad \mbox{for all } \quad \rho=0\,;
\]

\item[(5.iv)] for all $i=1,\dots,L$, there exist $\varepsilon >0$ and $K>0$ such that
\[
|f_i(\rho)| \leq K(1+|\rho|^{2+\varepsilon})\,;
\]
\item[(5.v)] for all $i=1,\dots,L$, the random function $Z^i$ is bounded and non-negative $\mathbb{P}-$a.s., that is
\[
Z^i(q,\omega) \in L^\infty(\QQ) \,,\quad \mbox{and} \quad Z^i(q,\omega) \geq 0\,, \quad \mbox{ for a.e. } q \in \QQ\,, \quad \mathbb{P}-\mbox{a.s.}\,,
\]
and has finite $p-$th moment, $p \geq 1$, that is
\[
\EE \left \| Z^i\right\|^p_\infty < \infty\,.
\]
\end{description}
\end{enumerate}
\end{Hypothesis}

\begin{Remark}
As mentioned in the Introduction, we will not focus on specific examples of chemical environments. Nonetheless, since almost any chemical model contains second-order reaction rates, we are forced to consider more general assumptions than the standard global Lipschitz condition. For this reason, we rather consider a mass control condition in Hypothesis \ref{HYP:3}(ii). Such assumptions can be immediately seen to hold in a reaction-diffusion description of the systems introduced in \cite{abolfath2020oxygen,labarbe2020physicochemical}. 
\demo
\end{Remark}

Therefore, the resulting process is characterized by the process defined in equation \eqref{EQN:SpatGSM2PI} with the addition of the chemical system as defined in equation \eqref{EQN:Chemistry}. We thus have the following representation
\begin{equation}\label{EQN:SpatGSM2PI2}
\begin{cases}
\norm{f^\XX,\nu^\XX(t)} &= \LL^\XX_B \nu^\XX(t) +\\
&+ \int_0^t \int_{\QQ^{|(\xi^\XX,\xi^\YY)|}} \int_{\NN^2} \int_{\RR}\int_{\RR}  \left[ \norm{f^\XX,\nu^{\XX}(s_-) + \sum_{i=1}^{\xi^\XX} \delta_{q_i^\XX}} - \norm{f^\XX,\nu^{\XX}(s_-)}\right] \times\\[3pt]
&\qquad \times \Ind{\theta_1 \leq \DG}\Ind{\theta_2 \leq \zeta_\xi^\XX(q|\xi^\XX,\rho)} \PN^{\DG}(ds,dq,d\xi^\XX,d\xi^\YY,d\theta_1,d\theta_2)\,,\\[5pt]
\norm{f^\YY,\nu^\YY(t)} &= \LL^\YY_B \nu^\YY(t) +\\
&+\int_0^t \int_{\QQ^{|(\xi^\XX,\xi^\YY)|}} \int_{\NN^2} \int_{\RR}  \int_{\RR} \left[ \norm{f^\YY,\nu^{\YY}(s_-) + \sum_{i=1}^{\xi^\YY} \delta_{q_i^\YY}} - \norm{f^\YY,\nu^{\YY}(s_-)}\right] \times\\[3pt]
&\qquad \times \Ind{\theta_1 \leq \DG} \Ind{\theta_2 \leq \zeta_\xi^\YY(q|\xi^\YY,\rho)} \PN^{\DG}(ds,dq,d\xi^\XX,d\xi^\YY,d\theta_1,d\theta_2)\,,\\
\rho_i(q, t) &= \rho_{0;i}(q) + \int_0^t \left( D_i \Delta_q \rho_i(q, s) + f_i(\rho)\right) ds + \sum_{s<t} \sum_{k=1}^{\infty} \ZZ^k(q) \Ind{s = \tau^k_{\DG}}\,.
\end{cases}
\end{equation}

Augment then the filtration with the processes defined in $(v')-(vi)$.

\begin{Definition}\label{DEF:DefNu3}
We say that $\nu(t) = (\nu^\XX(t),\nu^\YY(t))$ as defined in equations \ref{DEF:Nu1}--\ref{DEF:Nu2} is a \textit{spatial radiation-induced DNA damage repair model under protracted irradiation} if $\nu = \left(\nu(t)\right)_{t \in \RR_+}$ is $\left (\mathcal{F}_t\right )_{t \in \RR_+}-$adapted and for any $f^\XX$ and $f^\YY\in C^2_0(\QQ)$ equation \eqref{EQN:SpatGSM2PI2} holds $\mathbb{P}-$a.s.
\end{Definition}

We thus have the following well-posedness result.

\begin{Theorem}\label{THM:E!3}
Let $\nu^\XX_0$ and $\nu^\YY_0$ two random measures with finite $p-$th moment, $p \geq 1$, that is it holds
\begin{equation}\label{EQN:InitN3}
\EE \norm{\mathbf{1},\nu^\XX_0}^p < \infty\,,\quad \EE \norm{\mathbf{1},\nu^\YY_0}^p < \infty\,.
\end{equation}

Then, under Hypothesis \ref{HYP:1}-\ref{HYP:2}-\ref{HYP:3}, for any $T>0$ it exists a unique strong solution of the system \eqref{EQN:SpatGSM2} in $\DD \left ([0,T],\MM \times \MM \times \left( L^\infty(\QQ) \right)^L \right)$. Also, it holds
\begin{equation}\label{EQN:SupEst3}
\mathbb{E} \sup_{t \leq T} \norm{\mathbf{1},\nu (t)}^p  < \infty\,, \quad \mathbb{E} \sup_{t \leq T} \sup_{i=1\,,\dots,L} \left \| \rho_i\right\|^p_\infty < \infty\,.
\end{equation}
\end{Theorem}
\begin{proof}
The main step of the proof follows alike Theorems \ref{THM:E!}-\ref{THM:E!2}, noticing that between consecutive jump times $T_m$ and $T_{m+1}$, under Hypothesis \ref{HYP:3} using \cite[Theorem 1.1]{fellner2020global}, equation \eqref{EQN:Chemistry} admits a unique strong solution in $C\left ([T_m,T_{m+1});\left( L^\infty(\QQ) \right)^L \right)$. In fact, since $Z$ is bounded and non-negative $\mathbb{P}-$a.s. we can infer that for any jump time $T_m$ it holds
\[
\begin{split}
\rho_i(q,T_{m}) &:= \lim_{t \uparrow T_{m}}\rho_i(q,t) + \ZZ(q)\,,\\
\rho_i(q,T_{m},\omega) \in L^\infty(\QQ)\,,&\quad \mbox{and} \quad \rho_i(q,T_{m},\omega) \geq 0\,, \quad \mbox{ for a.e. } q \in \QQ\,,\quad \mathbb{P}-\mbox{a.s.}\,.
\end{split}
\]

The rest of the proof follows Theorems \ref{THM:E!}-\ref{THM:E!2}.
\end{proof}

\begin{Remark}
It is worth noticing that, using \cite[Theorem 1.1]{fellner2020global}, we can infer that, between consecutive jump times $T_m$ and $T_{m+1}$, it holds $\rho \in C\left ((T_m,T_{m+1});\left( C^2_0(\QQ) \right)^L \right)$. Nonetheless, since we cannot require $Z$ to be smooth, we cannot conclude that $\rho \in \mathcal{D}\left ((0,T);\left( C^2_0(\QQ) \right)^L \right)$ as at the jump times $T_m$, $\rho(T_m,q)$ can only be shown to be bounded.
\demo
\end{Remark}

\section{The large population limit}


In the following we assume the model coefficients depend on a parameter $\KK$ and study the behavior of the measure-valued solution $\nu^\KK = \left (\NXK,\NYK\right)$ studied in previous Sections as $\KK \to \infty$.

We thus consider a sequence of the initial measure so that $\frac{1}{\KK}\left (\NXK_0,\NYK_0\right) \to \left (u^\XX_0,u^\YY_0\right)$ and we assume that the rates of the model as introduced in Sections \ref{SEC:SpatGSM2}-\ref{SEC:ProtI} are rescaled as follows
\[
\begin{split}
\AG^\KK (q,v) := \AG\left(q,\frac{v}{\KK}\right )\,,\quad &\RG^\KK (q,v) := \RG\left(q,\frac{v}{\KK}\right )\,,\\
\BG^\KK (q_1,q_2,v) := \frac{1}{\KK}\BG\left(q_1,q_2,\frac{v}{\KK}\right )\,,\quad &\DG^\KK := \KK \DG\,.\\
\end{split}
\]

We thus aim at characterizing the limit as $\KK \to \infty$ of the rescaled measure
\begin{equation}\label{EQN:ResM}
u^\KK(t):=\frac{1}{\KK}\sum_{i=1}^{N(t)} \delta_{Q_i(t)}\delta_{s_i} = \frac{1}{\KK}\nu^\KK(t)\,.
\end{equation}
As above we also define for short the marginal distributions
\begin{equation}\label{DEF:ResM2}
u^{\XX;\KK}(t)(\cdot) := u^\KK(t)\left(\, \cdot\,,\XX\right)\,,\quad u^{\YY;\KK}(t)(\cdot) := u^\KK(t)\left(\, \cdot\,,\YY\right)\,.
\end{equation}

Analogous arguments to the one derived in Sections \ref{SEC:SpatGSM2}-\ref{SEC:ProtI} show that $\left (\NXK(t),\NYK(t)\right )_{t \geq 0}$ is a Markov process with infinitesimal generator of the form
\begin{equation}\label{EQN:InfGenAlldProtL}
\LL^\KK F_{(f^\XX,f^\YY)}(\nu) = \LL_d^\KK F_{(f^\XX,f^\YY)}(\nu) + \sum_{h \in \{\RG,\AG,\BG,\DG\}} \LL_{h}^\KK F_{(f^\XX,f^\YY)}(\nu)\,,    
\end{equation}
with
\begin{equation}\label{EQN:InfDiffK}
\begin{split}
\LL_d^\KK F_{(f^\XX,f^\YY)}(\nu) &= \norm{\LL_d^\XX f^\XX,\nu^\XX} \frac{\partial}{\partial x}  F \left (\norm{f^\XX,\nu^XX},\norm{f^\YY,\nu^XX}\right) + \norm{\LL_d^\YY f^\YY,\nu^\YY} \frac{\partial}{\partial y}  F \left (\norm{f^\XX,\nu},\norm{f^\YY,\nu}\right) + \\
&+ \norm{\nabla^T f^\XX\, \sigma^\XX \,\nabla f^\XX,\nu^\XX} \frac{1}{\KK}\frac{\partial^2}{\partial x^2}  F \left (\norm{f^\XX,\nu},\norm{f^\YY,\nu}\right) + \norm{\nabla^T f^\YY \, \sigma^\YY \,\nabla f^\YY,\nu^\YY} \frac{1}{\KK} \frac{\partial^2}{\partial y^2}  F \left (\norm{f^\XX,\nu},\norm{f^\YY,\nu}\right)\,,
\end{split}
\end{equation}
and
\begin{equation}\label{EQN:InfReactLK}
\begin{split}
\LL_{\RG}^\KK F_{(f^\XX,f^\YY)}(\nu) &= \KK \int_{\QQ} \RG(q,\nu)\left[F_{(f^\XX,f^\YY)}\left(\nu^\XX - \frac{1}{\KK}\delta_{q},\nu^\YY\right) - F_{(f^\XX,f^\YY)}(\nu) \right]\nu^\XX(dq)\,,\\
\LL_{\AG}^\KK F_{(f^\XX,f^\YY)}(\nu) &= \KK \int_{\QQ} \int_{\QQ} \AG(q,\nu)\left[F_{(f^\XX,f^\YY)}\left(\nu^\XX - \frac{1}{\KK} \delta_{q},\nu^\YY + \frac{1}{\KK} \delta_{\bar{q}}\right) - F_{(f^\XX,f^\YY)}(\nu) \right]m^{\AG}(\bar{q}|q)d\bar{q}\nu^\XX(dq)\,,\\
\LL_{\BG}^\KK F_{(f^\XX,f^\YY)}(\nu) &= \KK^2 \int_{\tilde{\QQ}^2} \int_{\QQ} \PG(q_1,q_2)\BG(q_1,q_2,\nu)\left[F_{(f^\XX,f^\YY)}\left(\nu^\XX - \frac{1}{\KK} \delta_{q_1} - \frac{1}{\KK} \delta_{q_2},\nu^\YY + \frac{1}{\KK} \delta_{\bar{q}}\right) - F_{(f^\XX,f^\YY)}(\nu) \right] \times \\
&\qquad \qquad \qquad \qquad \times m^{\BG}(\bar{q}|q_1,q_2)d\bar{q}\nu^\XX(dq_1)\nu^\XX(dq_2) + \\
&+\int_{\tilde{\QQ}^2} \left (1-\PG(q_1,q_2)\right) \BG(q_1,q_2,\nu)\left[F_{(f^\XX,f^\YY)}\left(\nu^\XX -\frac{1}{\KK} \delta_{q_1} -\frac{1}{\KK} \delta_{q_2},\nu^\YY\right) - F_{(f^\XX,f^\YY)}(\nu) \right] \nu^\XX(dq_1)\nu^\XX(dq_2)\,,\\
\LL_{\DG}^\KK F_{(f^\XX,f^\YY)}(\nu) &= \DG\int_{\NN^2} \int_{\QQ^{|(\xi^\XX,\xi^\YY)|}} \left[F_{(f^\XX,f^\YY)}\left (\nu^\XX+\frac{1}{\KK}\sum_{i=1}^{\xi^\XX}\delta_{q_i^\XX},\nu^\YY+\frac{1}{\KK}\sum_{i=1}^{\xi^\YY}\delta_{q_i^\YY}\right ) - F_{(f^\XX,f^\YY)}(\nu) \right] \times \\
&\qquad \qquad \qquad \qquad \times \zeta_\xi(\left . q^\XX,q^\YY \right| \xi^\XX,\xi^\YY) dp(\xi^\XX,\xi^\YY) dq^\XX\,dq^\YY \,.\\
\end{split}
\end{equation}

We thus can infer from Theorem \ref{THM:MartR2} the following martingale property for the rescaled system.

\begin{Lemma}\label{LEM:MartRK}
Consider $\KK \geq 1$ fixed, assume that Hypothesis \ref{HYP:1}-\ref{HYP:2} holds true and that $\nu^\XX_0$ and $\nu^\YY_0$ are two random measures independent with finite $p-$th moment, $p \geq 2$. Then, the processes $\Mar^\XX$ and $\Mar^\YY$ defined for $f^\XX\,,\,f^\YY \in C^2_0$ by
\begin{equation}\label{EQN:Mart31ab}
\begin{cases}
\Mar^\XX(t) &= \norm{f^\XX,u^{\XX;\KK}(t)} - \norm{f^\XX,u^{\XX;\KK}_0} - \int_0^t \norm{\LL^\XX_d f^\XX (x),u^{\XX;\KK}(s)}ds +\\[3pt]
&+ \int_0^t\int_{\QQ} \left[ \RG^\KK(q,\nu) + \AG^\KK(q,\nu)\right] f^\XX(q)u^{\XX;\KK}(s)(dq)ds+\\[3pt]
&+\int_0^t\int_{\tilde{\QQ}^2}\BG^\KK(q_1,q_2,\nu) (f^\XX(q_1)+f^\XX(q_2))u^{\XX;\KK}(s)(dq_1)u^{\XX;\KK}(s)(dq_2)ds+\\[3pt]
&-\int_0^t \int_{\NN} \int_{\tilde{\QQ}^{\xi^\XX}} \DG^\KK \left (\sum_{i=1}^{\xi^\XX}f^\XX(q_i^\XX) \right) \zeta_{\xi^\XX}(\left . q^\XX \right| \xi^\XX) dq^\XX  dp_{\XX}(\xi^\XX) ds\,,\\[5pt]
\Mar^\YY(t) &= \norm{f^\YY,u^{\YY;\KK}(t)} - \norm{f^\YY,u^{\YY;\KK}} - \int_0^t \norm{\LL^\YY_d f^\YY (x),u^{\YY;\KK}(s)}ds +\\[3pt]
&- \int_0^t \int_{\QQ} \AG^\KK(q,\nu) \int_{\QQ} m^{\AG}(\bar{q}|q) f^\YY(\bar{q})d\bar{q} u^{\XX;\KK}(s)(dq)ds+\\[3pt]
&- \int_0^t\int_{\tilde{\QQ}^2} \int_{\QQ} \PG(q_1,q_2) \BG^\KK(q_1,q_2,\nu) f^\YY(\bar{q}) m^{\BG}(\bar{q}|q_1,q_2)d\bar{q}u^{\XX;\KK}(s)(dq_1)u^{\XX;\KK}(s)(dq_2)ds+\\[3pt]
&-\int_0^t \int_{\NN} \int_{\tilde{\QQ}^{\xi^\YY}} \DG^\KK \left (\sum_{i=1}^{\xi^\YY} f^\YY(q_i^\YY) \right) \zeta_{\xi^\YY}(\left . q^\YY \right| \xi^\YY)dq^\YY dp_{\YY}(\xi^\YY) ds\,,
\end{cases}
\end{equation}
are c\'adl\'ag $L^2-$martingales starting at 0 with predictable quadratic variation given by
\begin{equation}\label{EQN:MartQuad1ab}
\begin{cases}
\norm{\Mar^\XX}(t) &= \frac{1}{\KK}\int_0^t \norm{\nabla^T f^\XX \sigma^\XX \nabla f^\XX , u^{\XX;\KK}}ds +\\[3pt]
&+ \frac{1}{\KK}\int_0^t \int_{\QQ} \left[ \RG^\KK(q,\nu) + \AG^\KK(q,\nu) \int_{\QQ} m^{\AG}(\bar{q}|q)d\bar{q}  \right] \left( f^\XX(q) \right)^2 u^{\XX;\KK}(s)(dq)ds+\\[3pt]
& +\frac{1}{\KK}\int_0^t\int_{\tilde{\QQ}^2} \int_{\QQ} \PG(q_1,q_2)\BG^\KK(q_1,q_2,\nu) (f^\XX(q_1)+f^\XX(q_2))^2 m^{\BG}(\bar{q}|q_1,q_2)d\bar{q}u^{\XX;\KK}(s)(dq_1)u^{\XX;\KK}(s)(dq_2)ds+\\[3pt]
& +\frac{1}{\KK}\int_0^t\int_{\tilde{\QQ}^2} \left (1-\PG(q_1,q_2)\right)\BG^\KK(q_1,q_2,\nu) (f^\XX(q_1)+f^\XX(q_2))^2 u^{\XX;\KK}(s)(dq_1)u^{\XX;\KK}(s)(dq_2)ds+\\[3pt]
&-\frac{1}{\KK}\int_0^t \int_{\NN} \int_{\tilde{\QQ}^{\xi^\XX}} \DG^\KK \left (\sum_{i=1}^{\xi^\XX}f^\XX(q_i^\XX) \right)^2 \zeta_{\xi^\XX}(\left . q^\XX \right| \xi^\XX) dq^\XX  dp_{\XX}(\xi^\XX) ds\,,\\[5pt]
\norm{\Mar^\YY}(t)  &= \frac{1}{\KK}\int_0^t \norm{\nabla^T f^\YY \sigma^\YY \nabla f^\YY , u^{\YY;\KK}}ds +\\[3pt]
&- \frac{1}{\KK}\int_0^t \int_{\QQ} \AG^\KK(q,\nu) \int_{\QQ} m^{\AG}(\bar{q}|q) \left( f^\YY(\bar{q})\right)^2 d\bar{q} u^{\XX;\KK}(s)(dq)ds+\\[3pt]
&-\frac{1}{\KK} \int_0^t\int_{\tilde{\QQ}^2} \int_{\QQ} \PG(q_1,q_2) \BG^\KK(q_1,q_2,\nu) \left( f^\YY(\bar{q})\right)^2 m^{\BG}(\bar{q}|q_1,q_2)d\bar{q}u^{\XX;\KK}(s)(dq_1)u^{\XX;\KK}(s)(dq_2)ds+\\[3pt]
&-\frac{1}{\KK}\int_0^t \int_{\NN} \int_{\tilde{\QQ}^{\xi^\YY}} \DG^\KK \left (\sum_{i=1}^{\xi^\YY} f^\YY(q_i^\YY) \right)^2 \zeta_{\xi^\YY}(\left . q^\YY \right| \xi^\YY)dq^\YY dp_{\YY}(\xi^\YY) ds\,.
\end{cases}
\end{equation}
\end{Lemma}
\begin{proof}
The proof is analogous to the proof of Theorem \ref{THM:MartR}.
\end{proof}

We assume the following.

\begin{Hypothesis}\label{HYP:4}
\begin{enumerate}
\setItemnumber{6}
\item rescaled system:
\begin{description}
\item[(6.i)] the initial measure $u_0^\KK = (u_0^{\XX;\KK},u_0^{\YY;\KK})$ converges in law for the weak topology on $\MMPP(\QQ) \times \MMPP(\QQ)$ to some deterministic finite measure $u_0 = (u_0^{\XX},u_0^{\YY}) \in \MMPP(\QQ) \times \MMPP(\QQ)$; also we assume $\sup_{\KK} \mathbb{E} \norm{\unit,u^{\KK}_0}^3 < \infty$;\\
\item[(6.ii)] all the parameters $\RG^\KK$, $\AG^\KK$ and $\BG^\KK$ are continuous on the corresponding space, that is either $\QQ \times \RR$ or $\QQ \times \QQ \times \RR$ and they are Lipschitz continuous w.r.t. the last variable, that is there exist positive constants $L_\RG$, $L_\AG$ and $L_\BG$ such that, for all $q\,,\,q_1\,,\,q_2 \in \QQ$ and $v_1\,,\,v_2\in \RR$ it holds
\[
\begin{split}
&\left |\RG(q,v_1)-\RG(q,v_2)\right| \leq L_\RG \left |v_1-v_2 \right |\,,\\
&\left |\AG(q,v_1)-\RG(q,v_2)\right| \leq L_\AG \left |v_1-v_2 \right |\,,\\
&\left |\BG(q_1,q_2,v_1)-\RG(q_1,q_2,v_2)\right| \leq L_\BG \left |v_1-v_2 \right |\,.\\
\end{split}
\]
\end{description}
\end{enumerate}
\end{Hypothesis}

Next is the main result of the present section.

\begin{Theorem}\label{THM:Conv}
Assume that Hypothesis \ref{HYP:1}-\ref{HYP:2}-\ref{HYP:4} holds and consider $u^\KK$ as defined in equation \eqref{EQN:ResM}. Then for all $T>0$, the sequence $\left(u^\KK\right)_{\KK \in \mathbb{N}}$ converges in law in $\mathcal{D}\left([0,T];\MMPP(\QQ)\right)$ to a deterministic continuous measure-valued function in $\mathcal{C}\left([0,T];\MMPP(\QQ)\right)$, which is the unique weak solution of the following non-linear integro-differential equation, $f^\XX\,,\,f^\YY \in C^2_0$,
\begin{equation}\label{EQN:DLimit}
\begin{cases}
\norm{f^\XX,u^\XX(t)} &= \norm{f^\XX,u^\XX_0} + \int_0^t \norm{\LL^\XX_d f^\XX (x),u^\XX(s)}ds +\\[3pt]
&- \int_0^t\int_{\QQ} \left[ \RG(q,u) + \AG(q,u)\right] f^\XX(q)u^\XX(s)(dq)ds+\\[3pt]
&-\int_0^t\int_{\tilde{\QQ}^2}\BG(q_1,q_2,u) (f^\XX(q_1)+f^\XX(q_2))u^\XX(s)(dq_1)u^\XX(s)(dq_2)ds+\\[3pt]
&+\int_0^t \int_{\NN} \int_{\tilde{\QQ}^{\xi^\XX}} \DG \left (\sum_{i=1}^{\xi^\XX}f^\XX(q_i^\XX) \right) \zeta_{\xi^\XX}(\left . q^\XX \right| \xi^\XX) dq^\XX  dp_{\XX}(\xi^\XX) ds\,,\\[5pt]
\norm{f^\YY,u^\YY(t)} &= \norm{f^\YY,u^\YY} + \int_0^t \norm{\LL^\YY_d f^\YY (x),u^\YY(s)}ds +\\[3pt]
&+ \int_0^t \int_{\QQ} \AG(q,u) \int_{\QQ} m^{\AG}(\bar{q}|q) f^\YY(\bar{q})d\bar{q} u^\XX(s)(dq)ds+\\[3pt]
&+ \int_0^t\int_{\tilde{\QQ}^2} \int_{\QQ} \PG(q_1,q_2) \BG(q_1,q_2,u) f^\YY(\bar{q}) m^{\BG}(\bar{q}|q_1,q_2)d\bar{q}u^\XX(s)(dq_1)u^\XX(s)(dq_2)ds+\\[3pt]
&+\int_0^t \int_{\NN} \int_{\tilde{\QQ}^{\xi^\YY}} \DG \left (\sum_{i=1}^{\xi^\YY} f^\YY(q_i^\YY) \right) \zeta_{\xi^\YY}(\left . q^\YY \right| \xi^\YY)dq^\YY dp_{\YY}(\xi^\YY) ds\,.
\end{cases}
\end{equation}
Further, $u$ satisfies
\begin{equation}\label{EQN:UEst}
\sup_{t \in [0,T]} \norm{\unit,u(t)}<\infty\,.
\end{equation}
\end{Theorem}

The steps of the proof of Theorem \ref{THM:Conv} are standard in the literature, \cite{isaacson2022mean,bansaye2015stochastic}, or also \cite{isaacson2022mean} for the case of a bimolecular reaction. Nonetheless, general results that include the case treated in the present paper are not available in the literature since \cite{isaacson2022mean} cannot include zero-th order reactions whereas other results do not include pairwise interaction. Therefore, the study of the large-population limit of a spatial model with pairwise interaction and random generation of lesions has never been considered in the literature, so existing results cannot be directly applied to the present case. As customary in literature, for the sake of readability the proof of Theorem \ref{THM:Conv} will be divided into four steps.

\subsection{Step 1: uniqueness of solution}

The first result concerns the proof of the uniqueness of the limiting process \eqref{EQN:DLimit}.

\begin{Theorem}\label{THM:UniqK}
There exists a unique solution to the equation \eqref{EQN:DLimit} in $\mathcal{C}\left([0,T];\MMPP(\QQ)\right)$.
\end{Theorem}
\begin{proof}
Arguments similar to the proof in Theorem \ref{THM:E!2} imply that neglecting negative terms and using Grownall's lemma, the following estimate holds
\[
\sup_{t \in [0,T]}\norm{\unit,u^\XX(t)} \leq C\,,\quad \sup_{t \in [0,T]}\norm{\unit,u^\YY(t)} \leq C\,.
\]

Consider then two different solutions $u_1=(u^\XX_1,u^\YY_1)$ and $u_2=(u^\XX_2,u^\YY_2)$ to the equation \eqref{EQN:DLimit}, satisfying 
\[
\sup_{t \in [0,T]} \norm{\unit,u^\XX_1(t)-u^\XX_2(t)}< C(T)< \infty\,.
\]
Denote by $T^\XX$ and $T^\YY$ the semigroups generated by the operators $\LL^\XX_d$ and $\LL^\YY_d$. We thus have for any bounded and measurable functions $f$ such that $\|f\|_\infty \leq 1$,
\begin{equation}\label{EQN:DLimitSem}
\begin{cases}
\norm{f^\XX,u^\XX(t)} &= \norm{T^\XX(t)f^\XX,u^\XX_0}+\\[3pt]
&- \int_0^t\int_{\QQ} \left[ \RG(q,u) + \AG(q,u) \int_{\QQ} \right] T^\XX(t-s)f^\XX(q)u^\XX(s)(dq)ds+\\[3pt]
&-\int_0^t\int_{\tilde{\QQ}^2}\BG(q_1,q_2,u) T^\XX(t-s)(f^\XX(q_1)+f^\XX(q_2))u^\XX(s)(dq_1)u^\XX(s)(dq_2)ds+\\[3pt]
&+\int_0^t \int_{\NN} \int_{\tilde{\QQ}^{\xi^\XX}} \DG \left (\sum_{i=1}^{\xi^\XX}f^\XX(q_i^\XX) \right) \zeta_{\xi^\XX}(\left . q^\XX \right| \xi^\XX) dq^\XX  dp_{\XX}(\xi^\XX) ds\,,\\[5pt]
\norm{f^\YY,u^\YY(t)} &= \norm{T^\YY(t)f^\YY,u^\YY} + \\[3pt]
&+ \int_0^t \int_{\QQ} \AG(q,u) \int_{\QQ} m^{\AG}(\bar{q}|q) T^\YY(t-s)f^\YY(\bar{q})d\bar{q} u^\XX(s)(dq)ds+\\[3pt]
&+ \int_0^t\int_{\tilde{\QQ}^2} \int_{\QQ} \PG(q_1,q_2) \BG(q_1,q_2,u) T^\YY(t-s)f^\YY(\bar{q}) m^{\BG}(\bar{q}|q_1,q_2)d\bar{q}u^\XX(s)(dq_1)u^\XX(s)(dq_2)ds+\\[3pt]
&+\int_0^t  \int_{\NN} \int_{\tilde{\QQ}^{\xi^\YY}} \DG \left (\sum_{i=1}^{\xi^\YY} f^\YY(q_i^\YY) \right) \zeta_{\xi^\YY}(\left . q^\YY \right| \xi^\YY)dq^\YY dp_{\YY}(\xi^\YY) ds\,.
\end{cases}
\end{equation}

We thus have
\begin{equation}\label{EQN:DLimitSemX}
\begin{cases}
&\left |\norm{f^\XX,u^\XX_1(t)-u^\XX_2(t)}\right| \leq \int_0^t\left|\int_{\QQ} \left[ \RG(q,u_1) + \AG(q,u_1) \right] T^\XX(t-s)f^\XX(q)\left( u_1^\XX(s)(dq)-u_2^\XX(s)(dq)\right)\right|ds+\\[3pt]
&+\int_0^t\left|\int_{\tilde{\QQ}^2} \BG(q_1,q_2,u_1) T^\XX(t-s)(f^\XX(q_1)+f^\XX(q_2))\left(u_1^\XX(s)(dq_1)u_1^\XX(s)(dq_q)-u_2^\XX(s)(dq_1)u_2^\XX(s)(dq_1)\right)\right|ds+\\[3pt]
&+\int_0^t\left| \int_{\NN} \int_{\tilde{\QQ}^{\xi^\XX}} \DG \left (\sum_{i=1}^{\xi^\XX}T^\XX(t-s)f^\XX(q_i^\XX) \right) \zeta_{\xi^\XX}(\left . q^\XX \right| \xi^\XX) dp_{\XX}(\xi^\XX) dq^\XX  \right|ds +\\[3pt]
&+\int_0^t\left|\int_{\QQ} \left[ \RG(q,u_2) + \AG(q,u_2) - \RG(q,u_1) - \AG(q,u_1)\int_{\QQ}\right] T^\XX(t-s)f^\XX(q)u_2^\XX(s)(dq)\right|ds+\\[3pt]
&+\int_0^t\left|\int_{\tilde{\QQ}^2}\left(\BG(q_1,q_2,u_2)-\BG(q_1,q_2,u_1)\right) T^\XX(t-s)(f^\XX(q_1)+f^\XX(q_2))u_2^\XX(s)(dq_1)u_2^\XX(s)(dq_1)\right|ds\,.
\end{cases}
\end{equation}

Since it holds that $\| T^\XX(t-s)f^\XX(q)\|_\infty \leq 1$, we have that
\begin{equation}\label{EQN:EstE!L}
\begin{split}
&\left |\left[ \RG(q,u_1) + \AG(q,u_1) \right] T^\XX(t-s)f^\XX(q)\right| \leq \bar{\RG} + \bar{\AG}(1+|v|) \leq C\,,\\
&\left |\BG(q_1,q_2,u_1) T^\XX(t-s)(f^\XX(q_1)+f^\XX(q_2))\right| \leq \bar{\BG}(1+|v|) \leq C\,,\\
&\left | \DG \left (\sum_{i=1}^{\xi^\XX}T^\XX(t-s)f^\XX(q_i^\XX)\right) \right| \leq C\,,\\
&\left| \left[\RG(q,u_2) + \AG(q,u_2) - \RG(q,u_1) - \AG(q,u_1)\right] T^\XX(t-s)f^\XX(q) \right|\leq C \sup_{\|f^\XX\|_\infty \leq 1}\norm{\unit,u_1(t)-u_2(t)}\,,\\
&\left|\left(\BG(q_1,q_2,u_2)-\BG(q_1,q_2,u_1)\right) T^\XX(t-s)(f^\XX(q_1)+f^\XX(q_2)) \right| \leq C \sup_{\|f^\XX\|_\infty \leq 1}\norm{\unit,u_1(t)-u_2(t)}\,.
\end{split}
\end{equation}

We therefore can infer, using estimates \eqref{EQN:EstE!L} in equation \eqref{EQN:DLimitSemX}, that
\[
\left |\norm{f^\XX,u^\XX_1(t)-u^\XX_2(t)}\right| \leq C \int_0^t \sup_{\|f^\XX\|_\infty \leq 1}\norm{f^\XX,u_1(s)-u_2(s)}ds\,.
\]
Using thus Gronwall's lemma, we can finally conclude that for all $t\leq T$,
\[
\sup_{\|f^\XX\|_\infty \leq 1}\norm{f^\XX,u_1(s)-u_2(s)}=0\,,
\]
from which the uniqueness follows.
\end{proof}

\subsection{Step 2: propagation of moments}

\begin{Lemma}\label{LEM:1}
Assume that Hypothesis \ref{HYP:1}-\ref{HYP:2}-\ref{HYP:4} holds, then for any $T>0$, there exists a constant $C:= C(T)>0$ which depends on $T$ such that the following estimates hold
\begin{equation}\label{EQN:Estimate1}
\sup_{\KK} \mathbb{E}\sup_{t \in [0,T]} \norm{\unit,u^{\XX,\KK}(t)}^3 \leq C \,,\quad \sup_{\KK} \mathbb{E}\sup_{t \in [0,T]} \norm{\unit,u^{\YY,\KK}(t)}^3 \leq C \,.
\end{equation}

In particular, we further have
\begin{equation}\label{EQN:Estimate2}
\sup_{\KK} \mathbb{E}\sup_{t \in [0,T]} \norm{f^\XX,u^{\XX,\KK}(t)}^3 \leq C \,,\quad \sup_{\KK} \mathbb{E}\sup_{t \in [0,T]} \norm{f^\YY,u^{\YY,\KK}(t)}^3 \leq C \,.
\end{equation}
\end{Lemma}
\begin{proof}
The proof follows from computations similar to what is done in the proof of Theorem \ref{THM:E!2} obtaining a similar estimate as in equation \eqref{EQN:Est4}. That is we have,
\[
\mathbb{E}\sup_{t \in [0,T \wedge \tau^\XX_n]} \norm{\unit,u^{\XX,\KK}(t)}^3 \leq C\,,\quad \mathbb{E}\sup_{t \in [0,T \wedge \tau^\YY_n]} \norm{\unit,u^{\YY,\KK}(t)}^3 \leq C\,.
\]

Taking the limit $\tau^\XX_n$ and $\tau^\YY_n$ as $\KK \to \infty$, Fatou's lemma implies equation \eqref{EQN:Estimate1}. Estimate \eqref{EQN:Estimate2} thus follows from the fact that $f^\XX$ and $f^\YY$ are bounded.
\end{proof}

\subsection{Step 3: tightness}

In the following, we will denote by $\Lambda^\KK$ the law of the process $u^\KK = (u^{\XX;\KK},u^{\YY,\KK})$. We then have the following.

\begin{Theorem}\label{THM:Tight}
Assume Hypothesis \ref{HYP:1}-\ref{HYP:2}-\ref{HYP:4} hold, then the sequence of laws $\left (\Lambda^\KK\right)_{\KK \in \mathbb{N}}$ on $\mathcal{D}\left([0,T];\MMPP\right)$ is tight when endowed with the vague topology.
\end{Theorem}
\begin{proof}
We equip $\MMPP$ with the vague topology; to prove the tightness of the sequence of laws $\Lambda^\KK$, using \cite{roelly1986criterion}, it is enough to show that the sequence of laws of the processes $\norm{f^\XX,u^{\XX;\KK}}$ and $\norm{f^\YY,u^{\YY;\KK}}$ are tight in $\mathcal{D}\left([0,T];\RR\right)$ for any function $f^\XX$ and $f^\YY$. As standard, in order to accomplish this, we employ the Aldous criterion \cite{aldous1978stopping} and Rebolledo criterion \cite{joffe1986weak}.

Notice first that, using the fact that $f^\XX$ and $f^\YY$ are bounded, we have already proved estimates \eqref{EQN:Estimate2}. Consider thus $\delta>0$ and two stopping times $(\tau_1,\tau_2)$ satisfying a.s. $0 \leq \tau_1 \leq \tau_2 \leq \tau_2+\delta \leq T$. Using Doob's inequality, together with estimate \eqref{EQN:Estimate1} and the martingale representation given in Lemma \ref{LEM:MartRK},
\[
\begin{split}
\mathbb{E}\left[\norm{\Mar^\XX}(\tau_2) - \norm{\Mar^\XX}(\tau_1)\right] &\leq C \delta \mathbb{E}\left[1+\sup_{t \in [0,T]} \norm{\unit,u^{\XX;\KK}}^3+\sup_{t \in [0,T]} \norm{\unit,u^{\YY;\KK}}^3 \right]\,,\\
\mathbb{E}\left[\norm{\Mar^\YY}(\tau_2) - \norm{\Mar^\YY}(\tau_1)\right] &\leq C \delta \mathbb{E}\left[1+\sup_{t \in [0,T]} \norm{\unit,u^{\XX;\KK}}^3+\sup_{t \in [0,T]} \norm{\unit,u^{\YY;\KK}}^3 \right]\,.
\end{split}
\]

Similarly, denoting by
\begin{equation}\label{EQN:FiniteVar}
\begin{cases}
A^\XX(t) &:=  \int_0^t \norm{\LL^\XX_d f^\XX (x),u^{\XX;\KK}(s)}ds +\\[3pt]
&- \int_0^t\int_{\QQ} \left[ \RG^\KK(q,\nu(s)) + \AG^\KK(q,\nu) \right] f^\XX(q)u^{\XX;\KK}(s)(dq)ds+\\[3pt]
&-\int_0^t\int_{\tilde{\QQ}^2} \BG^\KK(q_1,q_2,\nu) (f^\XX(q_1)+f^\XX(q_2))u^{\XX;\KK}(s)(dq_1)u^{\XX;\KK}(s)(dq_2)ds+\\[3pt]
&+\int_0^t \int_{\NN} \int_{\tilde{\QQ}^{\xi^\XX}} \DG^\KK \left (\sum_{i=1}^{\xi^\XX}f^\XX(q_i^\XX) \right) \zeta_{\xi^\XX}(\left . q^\XX \right| \xi^\XX) dq^\XX  dp_{\XX}(\xi^\XX) ds\,,\\[5pt]
A^\YY(t) &:= \int_0^t \norm{\LL^\YY_d f^\YY (x),u^{\YY;\KK}(s)}ds +\\[3pt]
&+ \int_0^t \int_{\QQ} \AG^\KK(q,\nu) \int_{\QQ} m^{\AG}(\bar{q}|q) f^\YY(\bar{q})d\bar{q} u^{\XX;\KK}(s)(dq)ds+\\[3pt]
&+ \int_0^t\int_{\tilde{\QQ}^2} \int_{\QQ} \PG(q_1,q_2) \BG^\KK(q_1,q_2,\nu) f^\YY(\bar{q}) m^{\BG}(\bar{q}|q_1,q_2)d\bar{q}u^{\XX;\KK}(s)(dq_1)u^{\XX;\KK}(s)(dq_2)ds+\\[3pt]
&+\int_0^t \int_{\NN} \int_{\tilde{\QQ}^{\xi^\YY}} \DG^\KK \left (\sum_{i=1}^{\xi^\YY} f^\YY(q_i^\YY) \right) \zeta_{\xi^\YY}(\left .q^\YY \right| \xi^\YY) d q^{\YY} dp_{\YY}(\xi^\YY) ds\,,
\end{cases}
\end{equation}
the finite variation part of $\norm{f^\XX,u^{\XX;\KK}(\tau_2)} - \norm{f^\XX,u^{\XX;\KK}(\tau_1)}$ and $\norm{f^\YY,u^{\YY;\KK}(\tau_2)} - \norm{f^\YY,u^{\YY;\KK}(\tau_1)}$ we have that
\[
\begin{split}
\mathbb{E}\left[\norm{A^\XX}(\tau_2) - \norm{A^\XX}(\tau_1)\right] &\leq C \delta \mathbb{E}\left[1+\sup_{t \in [0,T]} \norm{\unit,u^{\XX;\KK}}^3+\sup_{t \in [0,T]} \norm{\unit,u^{\YY;\KK}}^3 \right]\,,\\
\mathbb{E}\left[\norm{A^\YY}(\tau_2) - \norm{A^\YY}(\tau_1)\right] &\leq C \delta \mathbb{E}\left[1+\sup_{t \in [0,T]} \norm{\unit,u^{\XX;\KK}}^3+\sup_{t \in [0,T]} \norm{\unit,u^{\YY;\KK}}^3 \right]\,.\\
\end{split}
\]

The claim thus follows.
\end{proof}

\subsection{Step 4: identification of the limit}

\begin{Theorem}\label{THM:IdLim}
Assume Hypothesis \ref{HYP:1}-\ref{HYP:2}-\ref{HYP:4} hold, denote by $\Lambda$ the limiting law of the sequence of laws $\left (\Lambda^\KK\right)_{\KK \in \mathbb{N}}$. Denote by $u$ the process with law $\Lambda$. Then $u$ is a.s. strongly continuous and it solves equation \eqref{EQN:DLimit}.
\end{Theorem}
\begin{proof}
The fact that $u$ is a.s. strongly continuous follows from the fact that
\[
\sup_{t \in [0,T]}\sup_{\|f^\XX\|_{\infty} \leq 1}\left |\norm{f^\XX,u^{\XX;\KK}(t)} - \norm{f^\XX,u^{\XX;\KK}(t_-)} \right|\leq \frac{1}{\KK}\,.
\]

Consider $t\leq T$, $f^\XX$, $f^\YY$ and $\mathrm{u} = (\mathrm{u}^\XX,\mathrm{u}^\YY) \in \mathcal{D}\left ([0,T],\MMPP\times \MMPP\right)$; define the functionals
\[
\Psi^\XX : \mathcal{D}\left ([0,T],\MMPP\times \MMPP\right) \to \RR\,,
\]
as
\[
\begin{split}
&\Psi^\XX_{t,f^\XX} (\mathrm{u}^\XX,\mathrm{u}^\YY) = \norm{f^\XX,u^\XX(t)} - \norm{f^\XX,u^\XX_0} - \int_0^t \norm{\LL^\XX_d f^\XX (x),u^\XX(s)}ds +\\[3pt]
&+ \int_0^t\int_{\QQ} \left[ \RG(q,u) + \AG(q,u)\right] f^\XX(q)u^\XX(s)(dq)ds+\\[3pt]
&+\int_0^t\int_{\tilde{\QQ}^2}\BG(q_1,q_2,u) (f^\XX(q_1)+f^\XX(q_2))u^\XX(s)(dq_1)u^\XX(s)(dq_2)ds+\\[3pt]
&-\int_0^t \int_{\NN} \int_{\tilde{\QQ}^{\xi^\XX}} \DG \left (\sum_{i=1}^{\xi^\XX}f^\XX(q_i^\XX) \right) \zeta_{\xi^{\XX}}(\left . q^\XX\right| \xi^\XX)dq^{\XX} dp_{\XX}(\xi^\XX)  ds\,,\\[5pt]
&\Psi^\YY_{t,f^\YY} (\mathrm{u}^\XX,\mathrm{u}^\YY) = \norm{f^\YY,u^\YY(t)} - \norm{f^\YY,u^\YY} - \int_0^t \norm{\LL^\YY_d f^\YY (x),u^\YY(s)}ds +\\[3pt]
&- \int_0^t \int_{\QQ} \AG(q,u) \int_{\QQ} m^{\AG}(\bar{q}|q) f^\YY(\bar{q})d\bar{q} u^\XX(s)(dq)ds+\\[3pt]
&- \int_0^t\int_{\tilde{\QQ}^2} \int_{\QQ} \PG(q_1,q_2) \BG(q_1,q_2,u) f^\YY(\bar{q}) m^{\BG}(\bar{q}|q_1,q_2)d\bar{q}u^\XX(s)(dq_1)u^\XX(s)(dq_2)ds+\\[3pt]
&-\int_0^t \int_{\NN^2} \int_{\tilde{\QQ}^{\xi^\YY}} \DG \left (\sum_{i=1}^{\xi^\YY} f^\YY(q_i^\YY) \right) \zeta_{\xi^{\YY}}(\left .q^\YY \right|\xi^\YY) dq^\YY dp_{\YY}(\xi^\YY) ds\,.
\end{split}
\]

We aim to show that, for any $t \leq T$,
\[
\mathbb{E}\left | \Psi^\XX_{t,f^\XX} (\mathrm{u}^\XX,\mathrm{u}^\YY) \right| = \mathbb{E}\left |  \Psi^\YY_{t,f^\YY} (\mathrm{u}^\XX,\mathrm{u}^\YY) \right| = 0\,.
\]

Using Lemma \ref{LEM:MartRK} we have that
\[
\MM^\XX(t) = \Psi^\XX_{t,f^\XX} (\mathrm{u}^{\XX;\KK},\mathrm{u}^{\YY;\KK})\,,\quad \MM^\YY(t) = \Psi^\YY_{t,f^\YY} (\mathrm{u}^{\XX;\KK},\mathrm{u}^{\YY;\KK})\,.
\]

Using thus Lemma \ref{LEM:MartRK} together with estimate \eqref{EQN:Estimate1} and Hypothesis \ref{HYP:1}-\ref{HYP:2}-\ref{HYP:4}, we have that
\[
\begin{split}
&\mathbb{E} \left | \MM^{\XX;\KK}(t) \right|^2 = \mathbb{E} \left \langle \MM^{\XX;\KK} \right \rangle (t) \leq \frac{C}{\KK}\,,\\
&\mathbb{E} \left | \MM^{\YY;\KK}(t) \right|^2 = \mathbb{E} \left \langle \MM^{\YY;\KK} \right \rangle (t) \leq \frac{C}{\KK}\,,\\
\end{split}
\]
which goes to 0 as $\KK \to \infty$. Therefore, we have
\[
\begin{split}
&\lim_{\KK \to \infty}\mathbb{E} \left | \Psi^\XX_{t,f^\XX} (\mathrm{u}^{\XX;\KK},\mathrm{u}^{\YY;\KK})\right | = 0\,,\\
&\lim_{\KK \to \infty}\mathbb{E} \left | \Psi^\YY_{t,f^\XX} (\mathrm{u}^{\XX;\KK},\mathrm{u}^{\YY;\KK})\right | = 0\,.
\end{split}
\]

Since $\mathrm{u}$ is a.s. strongly continuous and the functions $f^\XX$ and $f^\YY$ are bounded, then the functionals $ \Psi^\XX_{t,f^\XX}$ and $ \Psi^\YY_{t,f^\YY}$ are a.s. continuous at $\mathrm{u}$. Also, for any $\mathcal{D}\left ([0,T],\MMPP\times \MMPP\right)$ we have
\[
\begin{split}
&\left | \Psi^\XX_{t,f^\XX} (\mathrm{u}^{\XX;\KK},\mathrm{u}^{\YY;\KK}) \right | \leq C \sup_{s \in [0,T]}\left(1 + \norm{\unit,u^\XX(s)}^2 + \norm{\unit,u^\YY(s)}^2 \right)\,,\\
&\left | \Psi^\YY_{t,f^\XX} (\mathrm{u}^{\XX;\KK},\mathrm{u}^{\YY;\KK})\right | \leq C \sup_{s \in [0,T]}\left(1 + \norm{\unit,u^\XX(s)}^2 + \norm{\unit,u^\YY(s)}^2 \right)\,.
\end{split}
\]

Therefore the sequences $\left( \Psi^\XX_{t,f^\XX} (\mathrm{u}^{\XX;\KK},\mathrm{u}^{\YY;\KK})\right)_{\KK \in \mathbb{N}}$ and $\left( \Psi^\YY_{t,f^\YY} (\mathrm{u}^{\XX;\KK},\mathrm{u}^{\YY;\KK})\right)_{\KK \in \mathbb{N}}$ are uniformly integrable so that
\[
\begin{split}
&\lim_{\KK \to \infty}\mathbb{E} \left | \Psi^\XX_{t,f^\XX} (\mathrm{u}^{\XX;\KK},\mathrm{u}^{\YY;\KK})\right | = \mathbb{E} \left | \Psi^\XX_{t,f^\XX} (\mathrm{u}^{\XX},\mathrm{u}^{\YY})\right |\,,\\
&\lim_{\KK \to \infty}\mathbb{E} \left | \Psi^\YY_{t,f^\YY} (\mathrm{u}^{\XX;\KK},\mathrm{u}^{\YY;\KK})\right | = \mathbb{E} \left | \Psi^\YY_{t,f^\YY} (\mathrm{u}^{\XX},\mathrm{u}^{\YY})\right |\,,\\
\end{split}
\]
which concludes the proof.
\end{proof}

\begin{Theorem}\label{THM:TightW}
Assume Hypothesis \ref{HYP:1}-\ref{HYP:2}-\ref{HYP:4} hold, then the sequence of laws $\left (\Lambda^\KK\right)_{\KK \in \mathbb{N}}$ on $\mathcal{D}\left([0,T];\MMPP\right)$ is tight when endowed with the weak topology.
\end{Theorem}
\begin{proof}
Using \cite{meleard1993convergences}, since the limiting process is continuous, repeating the above calculations with $f^\XX = f^\YY = 1$, we obtain that the sequences $\left ( \norm{\unit,u^{\XX;\KK}} \right )_{\KK}$ and $\left (\norm{\unit,u^{\YY;\KK}} \right )_{\KK}$ converge to $\norm{\unit,u^{\XX}}$ and to $\norm{\unit,u^{\YY}}$ in $\mathcal{D}\left ([0,T];\RR\right)$.
\end{proof}

\subsection{Step 5: proof of the convergence theorem}

\begin{proof}[proof of Theorem \ref{THM:Conv}]
Putting together Theorems \ref{THM:UniqK}-\ref{THM:IdLim}-\ref{THM:TightW}, the claim follows.
\end{proof}

\section{Conclusions}

In the present paper, we introduced a general stochastic model that describes the formation and repair of radiation-induced DNA damage. The derived model generalizes the previously studied model, \cite{cordoni2021generalized,cordoni2022cell,cordoni2022multiple}, including a spatial description, allowing for pairwise interaction of cluster damages that might depend on the distances between damages, as well as a general description of the effect of radiation on biological tissue under a broad range of irradiation condition. We studied the mathematical well-posedness of the system as well as we characterize the large system behavior. Further, we believe that the derived model could play a role in describing an effect of recent interest in radiobiology, called the \textit{FLASH effect}. In particular, at extremely high rates of particle delivery, it has been empirically seen that the unwanted effects of radiation on healthy tissue decrease while the killing effect on the tumor is maintained. Although numerous studies on the topic, the physical and biological mechanism at the very core of this effect is today largely unknown. It is nonetheless believed that spatial interactions of particles can play a major role in the origin of this effect. Therefore, the model introduced in the present research can have the generality to provide a stochastic description of the effect that spatial interdependence between particles can have on biological tissue.

\cleardoublepage
\bibliographystyle{apalike}

\bibliography{bib}

\end{document}